\documentclass[12pt]{article}

\usepackage{amsmath}
\usepackage{amsthm,amscd,amsfonts}
\usepackage{amssymb, upref, color}
\usepackage{amsmath,amssymb,amsthm,mathrsfs}
\usepackage{color}
\usepackage[dvips]{graphicx}
\usepackage{epsf,epsfig, verbatim} %subfigure,
\usepackage{latexsym,bm}

\numberwithin{equation}{section}

\theoremstyle{plain}

\begin{document}

\title{ Higher regularity of homeomorphisms in the Hartman-Grobman theorem for  semilinear evolution equations
\footnote{ This paper was jointly supported from the National Natural
Science Foundation of China under Grant (No. 11931016 and 11671176) and  Grant Fondecyt 1170466.}
}
\author
{
Weijie Lu$^{a}$\,\,\,\,\,
Manuel Pinto$^{b}$\,\, \,\,
Y-H Xia$^{a}\footnote{Corresponding author. Y-H. Xia, xiadoc@outlook.com;yhxia@zjnu.cn. Address: College of Mathematics and Computer Science,  Zhejiang Normal University, 321004, Jinhua, China}$
\\
{\small \textit{$^a$ College of Mathematics and Computer Science,  Zhejiang Normal University, 321004, Jinhua, China}}\\
{\small \textit{$^b$ Departamento de Matem\'aticas, Universidad de Chile, Santiago, Chile }}\\
{\small Email:luwj@zjnu.edu.cn;  pintoj.uchile@gmail.com;  yhxia@zjnu.cn.}
}

\maketitle

\begin{abstract}

     %This paper concerns the regularity of the conjugacy the Hartman-Grobman theorem of Hein and Pr\"{u}ss (2016) .
  Hein and Pr\"{u}ss \cite{Hein-Pruss1} presented a version of Hartman-Grobman type $C^{0}$ linearization result for semilinear hyperbolic evolution equations.
  They showed that the linearising map (homeomorphism) $H$ and its inverse $G=H^{-1}$ are H\"{o}lder continuous.
  {\bf An important question}: is it possible to improve the regularity of the homeomorphisms $H$?
  In the present paper, we first formulate the result that  the homeomorphisms $H$  in the Hartman-Grobman theorem is Lipchitzian, but the inverse $G$ is merely H\"{o}lder continuous.
We also give a generalized local linearization result in this paper. Finally, some applications end the paper.
    As pointed out by Backes et al. \cite{BDK-JDE}, even if the diffeomorphism $F$ is $C^{\infty}$, the conjugacy (homeomorphism) can fail to be locally Lipschitz. The homeomorphisms are in general only locally H\"older continuous. In fact,it is proved that in all the previous works on Hartman-Grobaman theorem in $C^0$ linearization, the homeomorphisms are H\"older continuous.  However, by establishing two effective dichotomy integral inequalities, we prove that the conjugacy is Lipchitzian, but the inverse is H\"{o}lder continuous. \\
    %Our result is the first one to observe the higher regularity of homeomorphisms in the Hartman-Grobman theorem.\\
{\bf keywords}: Hartman-Grobman theorem; semilinear evolution equations; homeomorphisms; linearization
    %claim this findings.
  %To the best of our knowledge, the findings appear nowhere before in the published literature.

\end{abstract}

\section{Introduction}
\subsection{Motivations and novelty}
   Recently, Hein and Pr\"{u}ss \cite{Hein-Pruss1} gave a version of   $C^{0}$ linearization result for semilinear hyperbolic evolution equation on Banach space. Hein and Pr\"{u}ss \cite{Hein-Pruss1} considered the following two evolution equations
\begin{eqnarray}\label{Hein-Pruss-eq}
\partial_{t}v=Av+r(v) \quad
\mathrm{and} \quad
\partial_{t}u=Au,
\end{eqnarray}
where $A$ is the generator of $C_{0}$-group $e^{At}$ on Banach space $X$ and $r: X\rightarrow X$ is bounded and Lipschitzian.
    Hein and Pr\"{u}ss \cite{Hein-Pruss1} successfully transferred the Hartman-Grobman theorem from ODEs to semilinear hyperbolic evolution equations. In fact, the general proof of finite dimension cannot directly be extended to the infinite dimension case.
    %It is worth mentioning that $C_{0}$-group is never compact except for the finite dimensional case, thus there is no intrinsic compactness available.    Hence,  the general proof of finite dimension cannot be extended to the infinite dimension case.

     Hein and Pr\"{u}ss \cite{Hein-Pruss1}  proved that both the homeomorphism $H$ and its inverse $G=H^{-1}$ are H\"older continuous.  They believed that their estimates of H\"older exponent are optimal. Higher regularity of the homeomorphisms seems to be an interesting and delicate question. A question is: Is it possible to improve the regularity of the homeomorphisms? This paper gives an positive answer.
    Therefore, in this paper, by establishing two effective dichotomy integral inequalities, we improve the regularity of the homeomorphisms.

  As pointed out by Backes et al. \cite{BDK-JDE}, even if the diffeomorphism $F$ is $C^{\infty}$, the conjugacy (homeomorphism) can fail to be locally Lipschitz. The homeomorphisms are in general only (locally) H\"older continuous (see \cite{BDK-JDE,B-V1,B-V2,B-V3,B-V4,Belitskii3,DZZ-PLMS, Hein-Pruss1,Shi-Zhang1, Zou-Xia}). However, we prove that the conjugacy (homeomorphism) is Lipchitzian, but the inverse is H\"{o}lder continuous. Our result is the first one to observe the higher regularity of homeomorphisms in the Hartman-Grobman theorem.

Now we summarize our goals  of this paper as follows.\\
 \noindent    {\bf The first purpose } is precisely to improve the regularity of the conjugacy.
    More specifically, if the mild solutions of semilinear system is bounded, then the linearising map $H$ can be Lipschitzian,
but its inverse $G$ is merely H\"{o}lder continuous.
    Without this premise, we say that they are both H\"{o}lder continuous.
\\
 \noindent   {\bf The second purpose} is to weaken an important assumption in Hein and Pr\"{u}ss (2016) \cite{Hein-Pruss1}.
    Hein and Pr\"{u}ss obtained the Hartman-Grobman theorem by setting that the whole $C_{0}$-group $e^{At}$ admits a dichotomy.
    In this paper, we reduce this assumption.
    In fact,  it is enough to assume that the $C_{0}$-group $e^{At}$ partially satisfies the exponential dichotomy.
    More specifically, equations (\ref{Hein-Pruss-eq}) can be rewritten as:
    \begin{eqnarray}
    \begin{cases}
     \partial_{t}v_{1}=Av_{1}+r(v_{1},v_{2}), \\
     \partial_{t}v_{2}=Bv_{2},
    \end{cases}
    \; \mathrm{and} \quad
    \begin{cases}
     \partial_{t}u_{1}=Au_{1}, \\
     \partial_{t}u_{2}=Bu_{2},
    \end{cases}
    \end{eqnarray}
where $A, B$ are the generators of $C_{0}$-groups $e^{At}, e^{Bt}$ on Banach spaces $X, Y$, respectively.
    The nonlinear term $r:X\times Y\to X$ is bounded and Lipschitz continuous.
    However, there is only required that $C_{0}$-group $e^{At}$ admits a dichotomy projection, but not requirement on $C_{0}$-group $e^{Bt}$.  \\
  \noindent  {\bf The third purpose} is to give a generalized version of local linearization.
    We suppose that the nonlinear term $r(x)$ admits a non-Lipschitz continuous with respect to $x$ on a small closed ball, that is:
    \[ |r(x_{1})-r(x_{2})|\leq \mathcal{L}(\max\{|x_{1}|,|x_{2}|\})|x_{1}-x_{2}|,\]
where $\mathcal{L}(\cdot): [0,\infty)\rightarrow [0,\infty)$ is a continuous, nondecreasing function and $\mathcal{L}(0)\equiv 0$.
    Hence, by a $C^{\infty}$ bump function, we obtain a local linearization result.\\
  \noindent  {\bf The fourth purpose} is to establish two effective dichotomy integral inequalities in Section 2.3, which has a better estimate than Gronwall (Bellman) inequalities.
    It will be a novel and powerful tool to help us deal with the properties of linearising maps. To show its advantage over Bellman inequality, we also use Bellman inequality to prove the regularity, but we obtain that both the homeomorphisms are H\"older continuous.\\
     \noindent  {\bf The last purpose} is to apply our results to some applications including  the Hodekin-Huxley equations for the nerve axon.

\subsection{Mechanism of our improvements}
    In general, to prove the regularity of the linearising map $H$, one employs the Bellman (Gronwall) inequality, see for examples \cite{Hein-Pruss1,Shi, Shi-Zhang1,Xia1,Zou-Xia}.
    %The advantage of the Bellman inequality is that it does not require the boundedness of the mild solution to the constrained system.
    %Thus, one can use Bellman inequality to get an estimate for any systems (bounded and unbounded).
    However, the disadvantage of the Bellman inequality is that it will result in an exponential estimate of the form $e^{\alpha t} (\alpha>0)$, one can refer
to Lemma \ref{Lemma-10} in this paper.
    It is expansive, and the expansive estimate leads us to prove that the homeomorphism is H\"{o}lder continuous, not Lipschitzian.
    Therefore, most of the previous works on the regularity of conjugacy  \cite{DZZ-PLMS,Hein-Pruss1, Shi,Shi-Zhang1,Xia1,Zgl-HG,Zou-Xia} is H\"{o}lder continuous.

    On the contrary, the advantage of the dichotomy integral inequality (see Lemma \ref{Lemma-9} in Section 2.3) is that it yields an exponential decay of the form
$e^{-\alpha_{1} t} (\alpha_{1}>0)$.
    Thus, by dichotomy inequality, we can prove the Lipschitz continuity of the homeomorphism $H$ due to a better estimate (the exponential decay).
   % However, the  disadvantage of the dichotomy inequality is that it requires the boundedness of the constrained system.
    %Therefore, if the solution of the constrained system is unbounded, the dichotomy inequality is invalid.
   % Consequently, it is impossible to prove Lipschitz continuous for unbounded system.
   However, if you use Bellman inequality, it is impossible to prove the Lipschitz continuity of the homeomorphisms due to the bad estimate $e^{\alpha t}$ with $\alpha>0$.

   %For the sake of comparison and showing the advantage of the dichotomy integral inequality over the Bellman inequalities, we prove that the homeomorphism and its inverse are H\"{o}lder continuous by using Bellman inequality.

\subsection{History of linearization}
    The classical Hartman-Grobman theorem \cite{Grobman1,Hartman1,Robinson} states that if $x^{*}$ is a hyperbolic equilibrium point of a $C^{1}$ vector field $F(x)$
with flow $\varphi_{t}(x)$,
then there exist a neighborhood $\mathcal{O}$ of $x^{*}$ such that $\varphi$ is topologically conjugated to its linearization on $\mathcal{O}$. The equivalent function $H$  in general is not in $C^{1}$ (see Chicone \cite{Chicone1}, Rodrigues and Sol\'a-Morales \cite{R-S-3}). Equivalently, it can be stated that if $x^{*}$ is a hyperbolic fixed point of a $C^{1}$ diffeomorphism $F:\mathbb{R}^n\rightarrow\mathbb{R}^n$,
then there exist a neighborhood $\mathcal{O}$ of $x^{*}$ such that $F$ on $\mathcal{O}$ is topologically conjugated to $DF(x^{*})$. It also has a global version. %Namely, if $A:\mathbb{R}^n\rightarrow\mathbb{R}^n$ is a hyperbolic automorphism and $f:\mathbb{R}^n\rightarrow\mathbb{R}^n$ is a bounded Lipschitz map with sufficiently small Lipschitz constant, then $A+f$ is topologically conjugated to $A$.
    Palmer \cite{Palmer1} firstly extended the global version of Hartman-Grobman theorem to the nonautonomous differential equations in  finite dimensional space. Palmer's linearization theorem states that
     if the nonautonomous linear system admits an exponential dichotomy, and the nonlinear perturbation is bounded and Lipschitzian,
then nonlinear system can be linearized. %Lu, Pinto and Xia \cite{Lu-Pinto-Xia} improved his result.
    To weaken the Palmer's conditions,
    Jiang \cite{Jiang1} presented a version of Hartman-Grobman theorem by setting that the linear system admits a generalized exponential dichotomy.
    Huerta \cite{Huerta2,Huerta1} constructed a topological conjugacy between linear system and an unbounded nonlinear perturbation,
while nonautonomous linear system admits a nonuniform contraction.
Barreira and Valls \cite{B-V1,B-V2,B-V3,B-V4} proved several versions of Hartman-Grobman theorem in different situations with the assumption that the linear systems admit a nonuniform dichotomy. Moreover, they proved that the topological conjugacy is H\"{o}lder continuous. Recently,  Backes et al. \cite{BDK-JDE} obtained a version of Hartman-Grobman theorem without the assumption that the linear system admits exponential dichotomy. Their results generalized those of Reinfelds and  \v{S}teinberga \cite{Reinfelds-IJPAM}. Moreover, Backes \cite{BDK-JDE} proved  the H\"{o}lder continuity of the topological conjugacy and its inverse. Zgliczy\'nski \cite{Zgl-HG} established a version of Hartman-Grobman theorem
for diffeomorphisms and ODEs by geometric proofs based on covering relations and cone
conditions \cite{Zgl-JDE}.
    Different versions of the Hartman-Grobman theorem have been presented for the differential equations with piecewise constant argument \cite{Papaschinopoulos-A,Pinto-A, Zou-Xia}, %the noninstantaneous impulsive differential equations \cite{LPX1},
dynamic systems on time scales \cite{Potzche1}, the instantaneous impulsive system \cite{Fenner-Pinto,Xia1}.
   % Moreover, other linearization theorem in finite dimensional case were reported in the monographs, see Chicone \cite{Chicone1}, Hale \cite{Hale1}, Shi and Zhang \cite{Shi-Zhang1}.

  Now we pay our attention to the Hartman-Grobman theorem in infinite dimensional space. In fact, some arguments (such as Brouwers fixed point theorem) for the finite-dimensional proofs  are not valid for the infinite dimensional case.
    In 1969, Pugh \cite{Pugh1} gave a global version of Hartman's theorem on Banach space by providing that $A$ is bounded operator.
  %  In fact, he proved the linearization theorem for diffeomorphisms.
    Later, Lu \cite{Lu1} successfully proved a Hartman-Grobman theorem for the scalar reaction-diffusion equations. Moreover, Bates and Lu \cite{Lu2} obtained a Hartman-Grobman theorem for Cahn-Hilliard equation and phase field equations. They proved the linearization of these semilinear partial differential equations
    based on the invariant manifold theory and the invariant folation theory.
     %studied the equivalence of nonlinear impulse differential equations on Banach space.
    A reduction theorem was proven in Reinfelds and Sermone \cite{Reinfelds1}. Belitskill \cite{Belitskii3} studied a hyperbolic diffeomorphism in a Banach space
and prove that  the diffeomorphism admits local $\alpha$-H\"older linearization under some conditions.
    %and this results allowed us to reduce the given system to a much simpler one.
  %Rodrigues and Sol\`{a}-Morales \cite{R-S-1} presented a linearization of class $C^{1}$ for contractions on Banach space.
%Great efforts also  on the linearization is on the $C^{k}$ diffeomorphisms.
In bad situation, $C^{0}$ linearization is not enough to observe
the dynamic behaviors, for instant, to distinguish the node from the focus. To this purpose,
   Sternberg \cite{Sternberg1,Sternberg2} initially investigated $C^{r}$
linearization for $C^{k}(1\leq r\leq k\leq \infty)$ diffeomorphisms. Elbialy \cite{ElBialy1} and Rodrigues and Sol\`{a}-Morales \cite{R-S-1} improved Hartman's result \cite{Hartman2} to Banach space.
   Recently, Zhang et al. \cite{ZWN-JFA} improved the lower bound of $\alpha$ to lower the condition of $C^{1}$ linearization for planar contractions. They showed that the derivatives of the transformations in their $C^{1}$ linearization are H\"{o}lder continuous and prove that the estimates for the H\"{o}lder exponent  can not be improved anymore.  Zhang {\it et al} \cite{ZWN-JDE}
obtained a set of sharpness conditions for the $C^1$ linearization of hyperbolic diffeomorphisms. They also proved that the $C^1$ linearization is actually a $C^{1,\beta}$ linearization and gave sharp estimates for $\beta$.
  Zhang et al. \cite{ZWN-MA} studied the sharp regularity of linearization for $C^{1,1}$ hyperbolic diffeomorphisms in a Banach space. Futher, the $\alpha$ H\"older linearization of hyperbolic diffeomorphisms with resonance were studied in \cite{ZWN-ETDS}.
%under a weaker band condition by a decomposition with invariant foliations.
  Zhang et al. \cite{ZWN-TAMS} proved that the local homemorphism $H(x)$ is differentiable at the fixed point for a $C^{1}$ diffeomorphism
$G(x)$ with $DG(x)$ being $\alpha$-H\"{o}lder continuous at the fixed point.
    Recently, Dragi\v{c}evi\'{c} et al. \cite{DZZ-PLMS} extend van Strien's result \cite{Strien} of simultaneously differentiable and H\"older linearization to
nonautonomous differential equations with a nonuniform exponential dichotomy. Dragi\v{c}evi\'{c} et al. \cite{ZWN-MZ} also stuied the smooth linearization of nonautonomous difference equations with a nonuniform dichotomy.
Some more delicate conditions for $C^{r}$ (or $C^{1}$)-smooth linearization  obtained by Sell \cite{Sell1},
Belitskill \cite{Belitskii1}, Rodrigues and Sol\`{a}-Morales \cite{R-S-2}.

\subsection{Notations and Basic concepts}
  Let $(X,|\cdot|_X)$ and $(Y,|\cdot|_Y)$ denote two arbitrary Banach spaces.
  For convenience, both norms $|\cdot|_X$ and $|\cdot|_Y$ will be denoted by $|\cdot|$.
  Let $\mathbb{J}\subseteq \mathbb{R}$ be any real interval. Define
\[ \mathbb{BC}(\mathbb{J}, X):=\{ x: \mathbb{J}\rightarrow X |
        x(t)\; \mathrm{is}\; \mathrm{continuous}\; \mathrm{and}\; \sup\limits_{t\in \mathbb{J}} |x(t)| < \infty \} \]
and $\|x\|:=\sup\limits_{t\in \mathbb{J}} |x(t)|$. Let $U$ be an open subset of $X$, and define
\[ \mathbb{BC}(U,X):=\{f: U\rightarrow X | f(x)\; \mathrm{is}\; \mathrm{continuous}\; \mathrm{and}\; \sup\limits_{x\in U}|f(x)|<\infty\}\]
and $|f|_{\infty}:=\sup\limits_{x\in U}|f(x)|$.

  Obviously, $(\mathbb{BC}(\mathbb{J}, X), \|x\|)$ and $(\mathbb{BC}(U,X), |f|_{\infty})$ are both Banach spaces with norms $\|\cdot\|$ and $|\cdot|_{\infty}$, respectively.

\newtheorem{myddef}{\bf{Definition}\rm}[section]
\begin{myddef}\label{homeomorphism}(Topological Conjugacy, \cite{Shi-Zhang1})
  $x'=\varphi(x)$ and $y'=\phi(y)$ are said to be topologically conjugated if there exists a homeomorphism $H$ of $X$ into $X$
such that $H$ sends the solution of $x'=\varphi(x)$ onto the solution of $y'=\phi(y)$.
\end{myddef}

\begin{myddef}\label{exponential-dichotomy}(Exponential Dichotomy, \cite{Pruss2,Pruss3,Pruss4})
     A projection $P_{+}\in \mathbb{B}(X)$ is said to be a dichotomy projection for the  $C_{0}$-semigroup $e^{At}$ in $X$, if there exist constants
   $k\geq 1, \alpha >0$ such that the following conditions are satisfied:\\
{\bf (S1)} $P_{+}e^{At}=e^{At}P_{+}$, for all $t\geq 0$;\\
{\bf (S2)} $|e^{At}P_{+}x|\leq ke^{-\alpha t}|P_{+}x|$, for all $x\in X, t\geq 0$;\\
{\bf (S3)} $e^{At}P_{-}$ extends to a $C_{0}$-group on $R(P_{-})$;\\
{\bf (S4)} $|e^{At}P_{-}x|\leq k e^{\alpha t}|P_{-}x|$, for all $x\in X, t\leq 0$;\\
where $\mathbb{B}(X)$ is a bounded linear operator on $X$, $P_{-}= I_{X}-P_{+}$, $I_{X}$ is the identity operator.\\
  Moreover, the $Green$ $kernel$ corresponding to the exponential dichotomy is denoted by
 \begin{equation*}
 G_{A}(t)=
 \begin{cases}
 \ \ e^{At}P_{+}, \qquad t\geq 0,\\
 -e^{At}P_{-}, \qquad t<0.
 \end{cases}
 \end{equation*}
\end{myddef}

%Many important properties can be obtained if $e^{At}$ is known to have a dichotomy projection.
%   Moreover, Theorem 4 in \cite{Pruss2} claims an important property,  it reads as follows.\\
%{\bf Theorem P} \cite{Pruss2}
%{\it
%    Let $e^{At}$ be a $C_{0}$-semigroup in the Banach space $X$. Then the following are equivalent:
%  \begin{description}
% \item[(P1)]  $e^{At}$ admits a dichotomy projection.
% \item[(P2)] For each $f\in BC(\mathbb{R},X)$, there is a unique mild solution $x\in BC(\mathbb{R},X)$ of
% \[ x'=Ax+f(t), \quad t\in\mathbb{R}. \]
% \item[(P3)] $S_{1}=\{\mu\in\mathbb{C}: |\mu|=1\}\subset \rho(e^{A})$.
% \end{description}
% }
\subsection{Outline of this paper}
    We organize this paper as follows:
    our main results are stated in Section 2, where we state the global linearization results and local linearization results, respectively.
    Some preliminary results are presented in Section 3.
    Rigorous proofs are given to show the regularity of the linearising maps in Section 4.
    Finally, several applications are given to demonstrate our results.

\section{Statement of main results}

  In the present paper, we consider the following semilinear evolution equations:
\begin{equation}\label{semilinear-eq}
\begin{cases}
\partial_{t}u_{1}=Au_{1}+f(u_{1},u_{2}), \\
\partial_{t}u_{2}=Bu_{2}, \\
u_{1}(0)=u_{10}, u_{2}(0)=u_{20},
\end{cases}
\end{equation}
where $A, B$ are the generators of $C_{0}$-semigroups $e^{At}, e^{Bt}$ on the Banach spaces $X, Y$, respectively.
 The nonlinear term $f:X\times Y\to X$ is Lipschitzian.
  It is well known that the Cauchy problem of equation (\ref{semilinear-eq}) has a unique $mild$ $solution$ on the half line
$\mathbb{R}^{+}$.
  In addition, if $e^{At}, e^{Bt}$ is even $C_{0}$-groups, then this unique $mild$ $solution$ exists globally, i.e., on $\mathbb{R}$.

\subsection{Main results on the global linearization}
  Now, we are ready to present our main results on the global linearization in this paper.
  Firstly, we present the result on the existence of {\bf topological conjugacy}.
\newtheorem{myythm}{\bf{Theorem}\rm}[section]
\begin{myythm}\label{Thm-Linearization}
  Let $A, B$ be the generators of $C_{0}$-groups $e^{At}, e^{Bt}$ on Banach spaces $X, Y$, respectively.
  Assume that $e^{At}$ admits an exponential dichotomy.
  If nonlinear term $f$ is bounded (denoted it by $|f|_{\infty}$), Lipschitzian (Lipschitz constant denoted it by $|f|_{Lip}$), and satisfies
\begin{equation}\label{third-linear-condition}
 4k\alpha^{-1}\cdot |f|_{Lip}<1.
\end{equation}
  Then system (\ref{semilinear-eq}) is  topologically conjugated to its linear equations
\begin{equation}\label{linear-eq}
 \begin{cases}
\partial_{t}v_{1}=Av_{1}, \\
\partial_{t}v_{2}=Bv_{2},\\
 v_{1}(0)=v_{10}, v_{2}(0)=v_{20}.
 \end{cases}
\end{equation}
  Moreover, the linearising map $H(\cdot)$ and its inverse $G(\cdot)$ satisfy: $H(u)-u \in\mathbb{ BC}(X), G(v)-v \in \mathbb{BC}(X)$, for any $u, v\in X$.
\end{myythm}

\newtheorem{myyyrem}{\bf{Remark}\rm}[section]
\begin{myyyrem}
    There is no requirement on the generator $B$ in Theorem \ref{Thm-Linearization}. It means that $B$ can be a non-hyperbolic operator.
\end{myyyrem}

\noindent
  Next theorem is for the {\bf regularity} of the transformation $H$ and its inverse $G$.
{\color{blue}
\begin{myythm}\label{Thm-Regulaity}
Assume that all the conditions of Theorem \ref{Thm-Linearization} hold.% Then we have the following conclusions:
%\begin{description}
%\item[($C_{1}$)]
  If there exist $\xi_1:=x-\bar{x}\in P_+ X$ and $\eta_1:=y-\bar{y}\in P_+ Y$
(or $\xi_2:=x-\bar{x}\in P_- X$ and $\eta_2:=y-\bar{y}\in P_- Y$),
then the transformation $H$ is Lipschitz continuous, but its inverse $G$  is H\"{o}lder continuous.
  More specifically, there exist positive constants $p_{1}, p_{2}>0$ and $0<q<1$ such that for $u=(x,y)^T$ and $\bar{u}=(\bar{x},\bar{y})^T$
  \begin{equation*}
  \begin{cases}
  |H(u)-H(\bar{u})|\leq p_{1}\cdot |u-\bar{u}|, \\
  |G(u)-G(\bar{u})|\leq p_{2}\cdot|u-\bar{u}|^{q}.
  \end{cases}
\end{equation*}
%\item[($C_{2}$)]
 %  The transformation $H$ and its inverse $G$  are both H\"{o}lder continuous,
%i.e., there exist positive constants $\tilde{p}_{1}, p_{2}>0$ and $0<q, \tilde{q}<1$  such that
%for $u=(x,y)^T$ and $\bar{u}=(\bar{x},\bar{y})^T$
%\begin{equation*}
 % \begin{cases}
 % |H(u)-H(\bar{u})|\leq \tilde{p}_{1}\cdot|u-\bar{u}|^{\tilde{q}}, \\
 % |G(u)-G(\bar{u})|\leq p_{2}\cdot|u-\bar{u}|^{q}.
 % \end{cases}
%\end{equation*}
%\end{description}
\end{myythm}
}

\begin{myyyrem}
  In general, it is well known that the linearising maps obtained by $C^{0}$ linearization are H\"{o}lder continuous, i.e., $C^{0,\alpha} (0<\alpha<1)$.
 Our result is the first one to observe that the linearising map $H$ is Lipschitzian, but
   the inverse $G=H^{-1}$ is merely H\"{o}lder continuous. The method is based on two important dichotomy inequalities and theory of stable (unstable) manifold.
Furthermore,  $G$ cannot be improved to be Lipschitzian due to the right side integral diverges in the proof, see Remark \ref{g-holder} for more detail.
\end{myyyrem}

\noindent
  For the sake of comparison, we also present a result on the H\"older continuity of the both linearising maps based on the Bellman inequalities.
{\color{blue}
\begin{myythm}\label{Thm-Regulaity2}
Assume that all the conditions of Theorem \ref{Thm-Linearization} hold.
  Then both the transformation $H$ and its inverse $G$  are both H\"{o}lder continuous,
i.e., there exist positive constants $\tilde{p}_{1}, p_{2}>0$ and $0<q, \tilde{q}<1$  such that
for $u=(x,y)^T$ and $\bar{u}=(\bar{x},\bar{y})^T$
\begin{equation*}
  \begin{cases}
  |H(u)-H(\bar{u})|\leq \tilde{p}_{1}\cdot|u-\bar{u}|^{\tilde{q}}, \\
  |G(u)-G(\bar{u})|\leq p_{2}\cdot|u-\bar{u}|^{q}.
  \end{cases}
\end{equation*}
\end{myythm}
}

\subsection{Local linearization}
  We present a $generalized$ version of  local linearization.
  It is well known that the classical local linearization can be achieved as long as the nonlinear term $f(x)$ satisfies:
(1) $f(0)=f'(0)=0$; (2) $f(x)$ has a small Lipschitz constant in some neighborhood of 0.
  In the present paper, we improve the second condition, that is, while $f(x)$ is not Lipschitz continuous,
we can still perform the local linearization.
  We claim this facts as follows:

  Let $f(u_{1},u_{2}):X\times Y\to X$ be continuous and $f(0,0)\equiv 0$.
  For $u_{1}, \bar{u}_{1}\in X$, and $u_{2}, \bar{u}_{2}\in Y$, assume that
\[ |f(u_{1},u_{2})-f(\bar{u}_{1}, \bar{u}_{2})|\leq L(\max\{|u_{1}|, |\bar{u}_{1}|\}, \max\{|u_{2}|, |\bar{u}_{2}|\})
  (|u_{1}-\bar{u}_{1}|+|u_{2}-\bar{u}_{2}|),\]
where $L(\cdot,\cdot): [0,+\infty)\times [0,+\infty)\rightarrow [0,+\infty)$ is continuous, nondecreasing and $L(0,0)\equiv 0$.
  In order to obtain a local linearization version, we shall first discuss the modified equation of equation (\ref{semilinear-eq}), i.e.,
\begin{equation}\label{semilinear-eq-modified}
\begin{cases}
\partial_{t}u_{1}=Au_{1}+f_{\delta}(u_{1},u_{2}), \\
\partial_{t}u_{2}=Bu_{2}, \\
u_{1}(0)=u_{10}, u_{2}(0)=u_{20},
\end{cases}
\end{equation}
where $\delta>0$ is a given positive constant.
  $f_{\delta}(u_{1},u_{2})$ is the modified nonlinearity of $f(u_{1},u_{2})$ and defined as follows:
\[ f_{\delta}(u_{1}, u_{2})=f\left(\psi\left(\frac{|u_{1}|^{2}}{\delta^{2}}\right)u_{1}, \psi\left(\frac{|u_{2}|^{2}}{\delta^{2}}\right)u_{2}\right),\]
where $\psi$ is a $C^{\infty}$ bump function, namely, $\psi(t)=1$ as $t\in[0,1]$; $0<\psi(t)<1$ as $t\in(1,2)$; $\psi(t)=0$ as $t\in [2,\infty)$ and
$\psi'(t)\leq 2$.
  Obviously, $\psi(|u_{1}|^{2}), \psi(|u_{2}|^{2})$ is a smooth bump function and
\[ \begin{split}
 \left|D_{u_{1}}\left(\psi\left(\frac{|u_{1}|^{2}}{\delta^{2}}\right)u_{1}\right)\right|
                       \leq& |u_{1}| \cdot \left|\psi'\left(\frac{|u_{1}|^{2}}{\delta^{2}}\right)\right| \cdot \frac{2|u_{1}|}{\delta^{2}}
                                                        +\psi\left(\frac{|u_{1}|^{2}}{\delta^{2}}\right) \\
                                                        \leq& 2\cdot\frac{2(\sqrt{2}\delta)^{2}}{\delta^{2}}+1=9, \\
  \left|D_{u_{2}}\left(\psi\left(\frac{|u_{2}|^{2}}{\delta^{2}}\right)u_{2}\right)\right| \leq& 9.
\end{split}\]
Hence, it is clear to see that the modified nonlinear term $f_{\delta}(u_{1}, u_{2})$ has the following properties:
\begin{description}
  \item[(1)] $f_{\delta}(u_{1}, u_{2})| \overline{\mathcal{B}}(0,\delta) \equiv f(u_{1}, u_{2})$;
            $f_{\delta}(u_{1}, u_{2})|\{u_{1}\in X, u_2\in Y| |u_{i}|\geq \sqrt{2}\delta, i=1,2\}\equiv 0$,
    where $\overline{\mathcal{B}}(0,\delta)$ is the closure of $\mathcal{B}(0,\delta)$ and $\mathcal{B}(0,\delta)$ is a spherical neighborhood.
  \item[(2)] $|f_{\delta}(u_{1}, u_{2})-f_{\delta}(\bar{u}_{1}, \bar{u}_{2})|\leq 9L(\sqrt{2}\delta,\sqrt{2}\delta)
    (|u_{1}-\bar{u}_{1}|+|u_{2}-\bar{u}_{2}|)$ for $u_{1}, \bar{u}_{1}\in X$ and $u_2, \bar{u}_2 \in Y$.
  \item[(3)] $|f_{\delta}(u_{1}, u_{2})|\leq 2\sqrt{2}\delta\cdot L(\sqrt{2}\delta,\sqrt{2}\delta)$.
\end{description}
Now, we are in a position to present the {\bf local linearization theorem}.

\begin{myythm}\label{Local-theorem}
 Let $A, B$ be the generators of $C_{0}$-groups $e^{At}, e^{Bt}$ on Banach spaces $X, Y$, respectively.
 Assume that $e^{At}$ admits an exponential dichotomy.
 Further, for a given constant $\delta>0$, if the nonlinearity $f_{\delta}$ satisfies properties {\bf (1)-(3)} and such that
\[ 36k\alpha^{-1}\cdot L(\sqrt{2}\delta,\sqrt{2}\delta) <1.\]
Then equation (\ref{semilinear-eq-modified}) is topologically conjugated to its linear parts in $\overline{\mathcal{B}}(0,\delta)$.
\end{myythm}

\subsection{Two important dichotomy integral inequalities}
    Next, we present a generalized version of dichotomy integral inequalities, which will play an important role in our main proofs.
    It will help us to prove Lipschitz continuity for the linearising maps.
    %Consequently, it improves the regularity of the conjugacy.

\newtheorem{mylem-l}{\bf{Lemma}\rm}[section]
\begin{mylem-l}\label{Dichotomy-inequality-1st}
  Assume that the function $T(t): [0,s]\rightarrow [0,\infty)$ is continuous and bounded for any $s\in (0,\infty]$.
  If there exist non-negative constants  $\alpha$ and $a_{i}, i=1,...,4$ such that  $a_{3}+a_{4}< \alpha$,
  then for any $t\in[0,s]$, the inequality
\begin{equation}\label{ineq1}
  T(t)\leq a_{1}+ a_{2}e^{-\alpha t}+ a_{3}\int_{0}^{t}e^{-\alpha(t-\tau)}T(\tau)d\tau+a_{4}\int_{t}^{s}e^{\alpha(t-\tau)}T(\tau)d\tau
\end{equation}
implies
$T(t)\leq (1-\varpi)^{-1}(a_{1}+a_{2}e^{-\alpha_{1}t}),$
where
\[\begin{split}
\varpi:=& \sup\limits_{t\in [0,s]}
\left(a_{3} \int_{0}^{t} e^{-\alpha(t-\tau)}d\tau+ a_{4} \int_{t}^{s}e^{\alpha(t-\tau)}d\tau\right)=(a_{3}+a_{4})/\alpha<1, \\
\alpha_{1}:=& \alpha-a_3 \cdot (1-\varpi)^{-1}.
\end{split}\]
\end{mylem-l}

\begin{proof}
   For any $t\geq 0$, suppose that the function
   \begin{equation}\label{(1)}
     \hat{T}(t)=T(t)-\hat{a}_{1}
   \end{equation}
satisfies the following inequality
\begin{equation}\label{(2)}
  \hat{T}(t)\leq \hat{a}_{2}e^{-\alpha t}+\hat{a}_{3} \int_{0}^{t}e^{-\alpha(t-\tau)}\hat{T}(\tau)d\tau
                      +\hat{a}_{4}\int_{t}^{s}e^{\alpha(t-\tau)}\hat{T}(\tau)d\tau,
\end{equation}
where $\hat{a}_{i}, i=1,...,4$ are undetermined constants. From \eqref{(1)} and \eqref{(2)}, we obtain
\[\begin{split}
T(t)-\hat{a}_{1} \leq & \hat{a}_{2}e^{-\alpha t}+\hat{a}_{3} \int_{0}^{t}e^{-\alpha(t-\tau)}T(\tau)d\tau
                      +\hat{a}_{4}\int_{t}^{s}e^{\alpha(t-\tau)}T(\tau)d\tau \\
                      &- \hat{a}_{1}\hat{a}_{3}\int_{0}^{t} e^{-\alpha(t-\tau)}d\tau- \hat{a}_{1}\hat{a}_{4} \int_{t}^{s}e^{\alpha(t-\tau)}d\tau,
\end{split}\]
that is
\begin{equation}\label{(3)}
  T(t)\leq \hat{a}_{1}(1-(\hat{a}_{3}+\hat{a}_{4})/\alpha)
                   + \hat{a}_{2}e^{-\alpha t}+\hat{a}_{3} \int_{0}^{t}e^{-\alpha(t-\tau)}T(\tau)d\tau
                      +\hat{a}_{4}\int_{t}^{s}e^{\alpha(t-\tau)}T(\tau)d\tau.
\end{equation}
Comparing \eqref{(3)} with \eqref{ineq1}, we see that
\[\begin{cases}
 a_{1}=\hat{a}_{1}(1-(\hat{a}_{3}+\hat{a}_{4})/\alpha),\\
 a_{2}=\hat{a}_{2}, a_{3}=\hat{a}_{3}, a_{4}=\hat{a}_{4}.
\end{cases}\]
Thus it follows from \eqref{(2)} and dichotomy inequality \cite{Pinto2} that
\[ \hat{T}(t)\leq \frac{a_{2}}{1-\varpi}e^{-(\alpha-\frac{a_{3}}{1-\varpi})t}= \frac{a_{2}}{1-\varpi}e^{-\alpha_{1} t}.\]
Therefore,
\[ T(t)\leq  (1-\varpi)^{-1}(a_{1}+a_{2}e^{-\alpha_{1}t}).\]
This completes the proof.
\end{proof}

\newtheorem{myyrem}{\bf{Remark}\rm}[section]
\begin{myyrem}\label{remmark-dichotomy-ieq}
   If $a_{1}\equiv 0$, then $T(t)\leq (1-\varpi)^{-1} a_{2}e^{-\alpha_{1}t}$ for all $t\geq 0$.
%That is $T(t)$ is an exponential stable semigroup.
 %  This shows that if a $C_{0}$-semigroup $T(t)$ is bounded in the form of (\ref{ineq1}), then $T(t)$ is an exponential stable semigroup.
   If $a_{1}:=a_{1}(t)=\bar{a}_{1}\cdot e^{-\bar{\alpha} t} (\bar{a}_{1}, \bar{\alpha}\geq 0)$, then
    $T(t)\leq (1-\varpi)^{-1} (\bar{a}_{1}+a_{2}) e^{-\max\{\bar{\alpha}, \alpha_{1} \}t}$.
   Clearly, our results are a generalized version of classical dichotomy inequality in \cite{Pinto2}.
\end{myyrem}

    Next, we give a dichotomy inequality of negative time, the conditions and proof are the same as
in Lemma \ref{Dichotomy-inequality-1st}.
\begin{mylem-l}\label{Dichotomy-inequality-2ed}
  Assume that the function $T(t): [s,0]\rightarrow [0,\infty)$ be continuous and bounded for any $s\in [-\infty,0)$.
  If there exist non-negative constants  $\alpha$ and $a_{i}, (i=1,...,4)$ such that $a_{3}+a_{4}< \alpha$,
  then for any $t\in[s,0]$, the inequality
\begin{equation*}
  T(t)\leq a_{1}+ a_{2}e^{\alpha t}+ a_{3}\int_{t}^{0}e^{\alpha(t-\tau)}T(\tau)d\tau+a_{4}\int_{s}^{t}e^{-\alpha(t-\tau)}T(\tau)d\tau
\end{equation*}
implies
$T(t)\leq (1-\varpi)^{-1} (a_{1}+a_{2}e^{\alpha_{1}t}).$
\end{mylem-l}

\section{Preliminary results}
\subsection{Non-trivial bounded mild solution}

We start with a fundamental lemma which is a key for the other lemmas. Idea follows from \cite{Pruss2}.
\newtheorem{mylemm}{\bf{Lemma}\rm}[section]
\begin{mylemm}\label{bounded-Zero-solution}
 Suppose that $e^{At}$ admits an exponential dichotomy, then the Cauchy problem of the evolution equation $u'=Au$ has no non-trivial bounded mild solutions.
\end{mylemm}
\begin{proof}
  Suppose that $u(t;0,u_{0})=e^{At}u_{0}$ is a bounded mild solution of $u'=Au$ with the initial condition $u(0)=u_{0}\in X$.
  We want to show that
   $u(t)\triangleq u(t;0,u_{0})\equiv 0$. Firstly, for $t\geq 0$,  $u(t+s)=e^{At}u(s),$ for all $s\in\mathbb{R}$. In particular, $u(s)=e^{At}u(s-t)$.
Hence,
\[ |P_{+}u(s)| \leq |e^{At}P_{+}|\cdot|P_{+}u(s-t)|\leq ke^{-\alpha t}\cdot |P_{+}u(s-t)|. \]
Let $s-t=n \cdot r$, where $n$ is a integer and $r>0$ is a constant.   Since $\|u\|<+\infty$, then
\[ |u(s-t)|=|n\cdot u(r)|\leq |n|\cdot |u(r)|, \]
which implies that $\sup\limits_{t\geq0, s\in\mathbb{R}}|u(s-t)|< +\infty$. Therefore,
\[ |P_{+}u(s)| \leq k e^{-\alpha t} \cdot \|u\| \rightarrow 0 \quad \mathrm{as }\quad t \rightarrow +\infty ,\]
that is, $P_{+}u(s)\equiv 0$.
Similarly,
\[|P_{-}u(s)|\leq |e^{-At}P_{-}|\cdot|P_{-}u(s+t)|\leq ke^{-\alpha t}\cdot \|u\| \rightarrow 0 \quad\mathrm{as} \quad t \rightarrow +\infty , \]
which implies that  $P_{-}u(s)\equiv 0$. Hence, $u(s)\equiv 0$, for all $s\in \mathbb{R}$, which implies that the bounded mild solution
$u(t;0,u_{0})$ of $u'=Au$  is always equal to 0.
\end{proof}

\subsection{Constructing the conjugacy  }

To construct the conjugacy in Theorem \ref{Thm-Linearization}, we divide the proof of Theorem \ref{Thm-Linearization} into several preliminary results as follows.

\noindent
For simplicity, let
  $\left( \begin{array}{c}
    U_{1}(t,0,u_{10},u_{20})\\
    U_{2}(t,0,u_{10},u_{20})\\
  \end{array} \right)$
be the $mild$ $solution$ of (\ref{semilinear-eq})
and
 $\left( \begin{array}{c}
    V_{1}(t,0,v_{10},v_{20})\\
    V_{2}(t,0,v_{10},v_{20})\\
 \end{array} \right)$
be the $mild$ $solution$ of (\ref{linear-eq}), where
\[\begin{split}
 U_{1}(t,0,u_{10},u_{20})=&e^{At}u_{10}+\int_{0}^{t} e^{A(t-s)}\cdot f( U_{1}(s,0,u_{10},u_{20}), U_{2}(s,0,u_{10},u_{20}))ds, \\
 U_{2}(t,0,u_{10},u_{20})=&e^{Bt}u_{20},\quad
 V_{1}(t,0,v_{10},v_{20})=e^{At}v_{10}, \quad V_{2}(t,0,v_{10},v_{20})=e^{Bt}v_{20}.
 \end{split}\]
 %which come from the variation of formula.
In what follows, we always suppose that all the conditions of Theorem \ref{Thm-Linearization} are satisfied.

\begin{mylemm}\label{Lemma-2}
For each fixed $(\xi,\eta)\in X\times Y$, the linear inhomogeneous evolution equation
\begin{equation}\label{inhomogeneous-eq}
\begin{cases}
z'=Az-f(U_{1}(t,0,\xi,\eta), U_{2}(t,0,\xi,\eta)),\\
z(0)=h(\xi,\eta)\in X,
\end{cases}
\end{equation}
has a unique bounded mild solution.
\end{mylemm}
\begin{proof}
For any fixed $(\xi,\eta)$, observe that a solution of \eqref{inhomogeneous-eq} is given by the convolution
\[\begin{split}
 z(t):=(G_{A}*f)(t)=&-\int_{\mathbb{R}}G_{A}(s)f(U_{1}(t-s,0,\xi,\eta), U_{2}(t-s,0,\xi,\eta))ds \\
 =&-\int_{\mathbb{R}}G_{A}(t-s)f(U_{1}(s,0,\xi,\eta), U_{2}(s,0,\xi,\eta))ds.
 \end{split}\]
We shall show that $z(t)$ is the unique bounded mild solution of equation \eqref{inhomogeneous-eq}.
Since $f$ is bounded, we have
\[ \|z\|\leq 2k\alpha^{-1}\cdot |f|_{\infty}< +\infty. \]
This shows that $z(t)$ is bounded. From Lemma \ref{bounded-Zero-solution}, we know, for its linear homogeneous part $z'=Az$,
that the bounded mild solution is a zero solution.
Therefore,  $z(t)$ is a unique bounded mild solution.
\\
%We denote this bounded mild solution by $z(t,\xi,\eta)$, since it depends on $(\xi,\eta)$.
In particular, if we take $t=0$, then
\[ h(\xi,\eta)=z(0)= -\int_{\mathbb{R}}G_{A}(-s)f(U_{1}(s,0,\xi,\eta), U_{2}(s,0,\xi,\eta))ds.\]
This mild solution $h$ is also bounded with $\|h(\xi,\eta)\|\leq 2k\alpha^{-1}\cdot |f|_{\infty}$.
\end{proof}

\begin{mylemm}\label{Lemma-3}
For each fixed $(\xi,\eta)\in X\times Y$, the semilinear evolution equation
 \begin{equation}\label{nonlinear-eq-lemma3}
 \begin{cases}
 w'=Aw+f(V_{1}(t,0,\xi,\eta)+w, V_{2}(t,0,\xi,\eta)),\\
 w(0)=g(\xi,\eta)\in X,
 \end{cases}
 \end{equation}
has a unique bounded mild solution.
\end{mylemm}
\begin{proof}
Note that
$\mathbb{BC}(\mathbb{R},X):=\{w(t)\in X| w(t)\; \mathrm{is}\; \mathrm{continuous}\; \mathrm{and}\; \sup\limits_{t\in \mathbb{R}} |w(t)| < \infty \}$,
and define a map $\mathcal{T}$ as follows
\[ (\mathcal{T}w)(t):=\int_{\mathbb{R}}G_{A}(t-s)f(V_{1}(s,0,\xi,\eta)+w(s), V_{2}(s,0,\xi,\eta))ds.\]
{\bf Step 1.} We shall show that $\mathcal{T}$ is a contraction map on $\mathbb{BC}(\mathbb{R},X)$, consequently, $\mathcal{T}$ has a unique fixed point.
In fact, it is easy to obtain that
\[ \|\mathcal{T}w\|\leq 2k\alpha^{-1}\cdot |f|_{\infty}. \]
Hence, $\mathcal{T}$ is a self-map from $\mathbb{BC}(\mathbb{R},X)$ to $\mathbb{BC}(\mathbb{R},X)$.
Note that $f$ is Lipschitz continuous with Lipschitz constant $|f|_{Lip}$, we have
\[ |(\mathcal{T}w_{1}-\mathcal{T}w_{2})(t)|\leq \int_{\mathbb{R}} ke^{\alpha|t-s|}\cdot |f|_{Lip}\cdot |(w_{1}(s)-w_{2}(s))|ds \]
It follows from (\ref{third-linear-condition}) that
$\|\mathcal{T}w_{1}-\mathcal{T}w_{2}\| \leq 2k\alpha^{-1}\cdot |f|_{Lip}\cdot \|w_{1}-w_{2}\|< \frac{1}{2}\|w_{1}-w_{2}\|$.
Thus the map $\mathcal{T}$ has a unique fixed point, namely, $w^{*}=\mathcal{T}w^{*}$, and satisfying
\[ w^{*}(t)= \int_{\mathbb{R}}G_{A}(t-s)f(V_{1}(s,0,\xi,\eta)+w^{*}(s), V_{2}(s,0,\xi,\eta))ds. \]
Clearly, $w^{*}(t)$ is a bounded mild solution of (\ref{nonlinear-eq-lemma3}).\\
{\bf Step 2.} To prove uniqueness, suppose that there is another bounded mild solution $w^{+}(t)$ of (\ref{nonlinear-eq-lemma3}). Then
 \[ \begin{split}
  w^{+}(t)=& e^{At}w_{0}+\int_{0}^{t} f(V_{1}(t,0,\xi,\eta)+w^{+}(s), V_{2}(t,0,\xi,\eta))ds\\
                 =& e^{At}\left(w_{0}-\int_{\mathbb{R}}G_{A}(-s)f(V_{1}(s,0,\xi,\eta)+w^{+}(s), V_{2}(s,0,\xi,\eta))ds\right)\\
                  &+ \int_{\mathbb{R}}G_{A}(t-s)f(V_{1}(s,0,\xi,\eta)+w^{+}(s), V_{2}(s,0,\xi,\eta))ds.
 \end{split}\]
Since $\int_{\mathbb{R}}G_{A}(-s)f(V_{1}(s,0,\xi,\eta)+w^{+}(s), V_{2}(s,0,\xi,\eta))ds$ is bounded, it is convergent, denoted by $w_{0}^{+}$.
It is easy to see that $e^{At}(w_{0}-w_{0}^{+})$ is a bounded mild solution of $w'=Aw$ with the initial condition $w_{0}-w_{0}^{+}\in X$.
From Lemma \ref{bounded-Zero-solution}, it is a zero solution. Therefore,
\[  w^{+}(t)= \int_{\mathbb{R}}G_{A}(t-s)f(V_{1}(s,0,\xi,\eta)+w^{+}(s), V_{2}(s,0,\xi,\eta))ds. \]
Calculating $w^{+}(t)-w^{*}(t)$, we have
\[ |w^{+}(t)-w^{*}(t)| \leq \int_{\mathbb{R}}ke^{\alpha|t-s|}\cdot |f|_{Lip}\cdot |w^{+}(s)-w^{*}(s)|ds.   \]
It follows from (\ref{third-linear-condition}) that $\|w^{+}-w^{*}\|\leq 2k\alpha^{-1}\cdot |f|_{Lip}\cdot \|w^{+}-w^{*}\|< \frac{1}{2}\|w^{+}-w^{*}\|.$
This implies that $w^{+}(t)\equiv w^{*}(t)$, and the bounded mild solution is unique. \\
In particular, if we take $t=0$, then $g(\xi,\eta)=w(0)$ is bounded with $\|g(\xi,\eta)\|\leq 2k\alpha^{-1}\cdot |f|_{\infty}$.
\end{proof}

\begin{mylemm}\label{Lemma-4}
Let
$\left( \begin{array}{c}
    u_{1}(t)\\
    u_{2}(t)\\
\end{array} \right)$
be any mild solution of (\ref{semilinear-eq}). Then the semilinear evolution equation
\begin{equation}\label{nonlinear-eq-in-lemma4}
z'=Az+f(u_{1}+z,u_{2})-f(u_{1},u_{2})
\end{equation}
has a unique bounded mild solution $z=0$.
\end{mylemm}
\begin{proof}
The proof is similar to that in Lemma \ref{Lemma-3}.
\end{proof}

\noindent
Now, we define {\bf two maps} $H, G: X\times Y\to X$ as follows:
\[H(u_{1},u_{2})=\left( \begin{array}{c}
    H_{1}(u_{1},u_{2}) \\
    H_{2}(u_{1},u_{2}) \\
\end{array} \right)
=
\left( \begin{array}{c}
    u_{1}+h(u_{1},u_{2})  \\
    u_{2}  \\
\end{array} \right),\]
\[G(v_{1},v_{2})=\left( \begin{array}{c}
    G_{1}(v_{1},v_{2}) \\
    H_{2}(v_{1},v_{2}) \\
\end{array} \right)
=
\left( \begin{array}{c}
    v_{1}+g(v_{1},v_{2})  \\
    v_{2}  \\
\end{array} \right),\]
where $u_{1}, v_{1}\in X$ and $u_2, v_2\in Y$.

\begin{mylemm}\label{Lemma-5}
$\left( \begin{array}{c}
    H_{1}(U_{1}(t,0,u_{10},u_{20}),U_{2}(t,0,u_{10},u_{20})) \\
    H_{2}(U_{1}(t,0,u_{10},u_{20}),U_{2}(t,0,u_{10},u_{20})) \\
\end{array} \right)$
is a mild solution of (\ref{linear-eq}).
\end{mylemm}
\begin{proof}
It is clear that
\[H_{1}(U_{1},U_{2})=U_{1}+h(U_{1},U_{2}),\quad H_{2}(U_{1},U_{2})=e^{Bt}u_{20},\]
and $H_{2}$ is a mild solution of the second equation in \eqref{linear-eq}.
Thus we only show that $H_{1}$ is a mild solution of the first equation in \eqref{linear-eq}.
Since $U_{1}$ is a mild solution of $u_{1}'=Au_{1}+f(u_{1},u_{2})$,
and $h(U_{1},U_{2})$ is a mild solution of $z'=Az-f(u_{1},u_{2})$, we have that $H_{1}$ is a mild solution of
\[ y'=u'_{1}+z'=Au_{1}+f(u_{1},u_{2})+Az-f(u_{1},u_{2})=A(u_{1}+z)=Ay.\]
\end{proof}

\begin{mylemm}\label{Lemma-6}
$\left( \begin{array}{c}
    G_{1}(V_{1}(t,0,v_{10},u_{20}),V_{2}(t,0,v_{10},v_{20})) \\
    G_{2}(V_{1}(t,0,v_{10},u_{20}),V_{2}(t,0,v_{10},v_{20})) \\
\end{array} \right)$
is a mild solution of (\ref{semilinear-eq}).
\end{mylemm}
\begin{proof}
The proof is similar to that of Lemma \ref{Lemma-5}.
\end{proof}

\begin{mylemm}\label{Lemma-7}
For any fixed $v_{1}\in X$ and $v_2\in Y$,
$\left( \begin{array}{c}
    H_{1}(G_{1}(v_{1},v_{2}),G_{2}(v_{1},v_{2})) \\
    H_{2}(G_{1}(v_{1},v_{2}),G_{2}(v_{1},v_{2})) \\
\end{array} \right)$
=
$\left( \begin{array}{c}
    v_{1} \\
    v_{2} \\
\end{array} \right)$.
\end{mylemm}
\begin{proof}
 Let $(v_{1}(t),v_{2}(t))^{T}$ be any solution of (\ref{linear-eq}). From Lemma \ref{Lemma-6},
$\left( \begin{array}{c}
    G_{1}(v_{1}(t),v_{2}(t)) \\
    G_{2}(v_{1}(t),v_{2}(t))\\
\end{array} \right)$
is a solution of (\ref{semilinear-eq}). Moreover, by Lemma \ref{Lemma-5},
$\left( \begin{array}{c}
    H_{1}(G_{1}(v_{1}(t),v_{2}(t)),G_{2}(v_{1}(t),v_{2}(t))) \\
    H_{2}(G_{1}(v_{1}(t),v_{2}(t)),G_{2}(v_{1}(t),v_{2}(t))) \\
\end{array} \right)$
is a solution of (\ref{linear-eq}), denoted by $(\hat{v}_{1}(t),\hat{v}_{2}(t))^{T}$.
Since
\[H_{2}(G_{1}(v_{1}(t),v_{2}(t)),G_{2}(v_{1}(t),v_{2}(t)))= G_{2}(v_{1}(t),v_{2}(t))=v_{2}(t),\]
we obtain that $\hat{v}_{2}(t)=v_{2}(t)$. Let $J(t)=\hat{v}_{1}(t)-v_{1}(t)$, then $J(t)$ is also a mild solution of $z'=Az$.
It follows from the definition of $H, G$ that
\[ \begin{split}
|J(t)|=&|H_{1}(G_{1}(v_{1}(t),v_{2}(t)),G_{2}(v_{1}(t),v_{2}(t)))-v_{1}(t)| \\
          \leq & |H_{1}(G_{1}(v_{1}(t),v_{2}(t)),G_{2}(v_{1}(t),v_{2}(t)))-G_{1}(v_{1}(t),v_{2}(t))|\\
          &+ |G_{1}(v_{1}(t),v_{2}(t))-v_{1}(t)|.
\end{split}\]
From Lemma \ref{Lemma-2} and Lemma \ref{Lemma-3}, we have that
$\|J\| \leq 2k\alpha^{-1}\cdot |f|_{\infty}+2k\alpha^{-1}\cdot |f|_{\infty}=4k\alpha^{-1}\cdot |f|_{\infty}$.
Thus in view of Lemma \ref{bounded-Zero-solution}, $J(t)\equiv 0$, namely, $\hat{v}_{1}(t)=v_{1}(t)$.
Since  $(v_{1}(t),v_{2}(t))^{T}$ is an arbitrary solution of (\ref{linear-eq}),  Lemma \ref{Lemma-7} holds.
\end{proof}

\begin{mylemm}\label{Lemma-8}
For any fixed $u_{1}\in X$ and $u_2\in Y$,
$\left( \begin{array}{c}
    G_{1}(H_{1}(u_{1},u_{2}),H_{2}(u_{1},u_{2})) \\
    G_{2}(H_{1}(u_{1},u_{2}),H_{2}(u_{1},u  _{2})) \\
\end{array} \right)$
=
$\left( \begin{array}{c}
    u_{1} \\
    u_{2} \\
\end{array} \right)$.
\end{mylemm}
\begin{proof}
Combined with Lemma \ref{Lemma-4}, the proof of  Lemma \ref{Lemma-8} is similar to Lemma \ref{Lemma-7}.
\end{proof}

\subsection{Key lemma to prove the regularity}

Next lemma is a key lemma to prove the regularity of the conjugacy.
{\color{blue}
\begin{mylemm}\label{Lemma-9}
Let $(U_{1}(t,0,u_{10},u_{20}),U_{2}(t,0,u_{10},u_{20}))$ be a  mild solution of (\ref{semilinear-eq}).\\
(i) For any $\xi_1\in P_{+}(X),\eta_1\in P_{+} (Y)$,
 then (\ref{semilinear-eq}) has a unique bounded solution $(U_{1}(t),U_{2}(t))$
 satisfying for $t\geq 0$, $P_{+}(U_1(0),U_2(0))=(\xi_1,\eta_1)$; \\%, where $M_B>0$ and $\varrho>0$ are constants.\\
(ii) For any $\xi_2\in P_{-}(X),\eta_2\in P_{-} (Y)$,
 then (\ref{semilinear-eq}) has a unique bounded solution $(U_{1}(t),U_{2}(t))$
 satisfying for $t\leq 0$, $P_{-}(U_1(0),U_2(0))=(\xi_2,\eta_2)$.\\
  Moreover,  if exponent $\varpi=(2k/\alpha)\cdot |f|_{Lip}<1$ and $\alpha_{1}=\alpha-k|f|_{Lip}/(1-\varpi)>0$,
then the following conclusions hold:\\
(1) for any $(u_{10}-\bar{u}_{10})\in P_{+}(X)$ and $(u_{20}-\bar{u}_{20})\in P_{+}(Y)$, we deduce that
\begin{equation}\label{dichotomy-application}
\begin{split}
&|U_{1}(t,0,u_{10},u_{20})-U_{1}(t,0,\bar{u}_{10},\bar{u}_{20})|+|U_{2}(t,0,u_{10},u_{20})-U_{2}(t,0,\bar{u}_{10},\bar{u}_{20})| \\
\leq & (1-\varpi)^{-1} \cdot(k e^{-\alpha_{1} t}\cdot|P_{+}(u_{10}-\bar{u}_{10})|+ M_{B}|P_{+}(u_{20}-\bar{u}_{20})|), \quad t\geq 0;
\end{split}
\end{equation}
(2) for any $(u_{10}-\bar{u}_{10})\in P_{-}(X)$ and $(u_{20}-\bar{u}_{20})\in P_{-}(Y)$, we have that
\begin{equation}\label{dichotomy-application-2}
\begin{split}
 &|U_{1}(t,0,u_{10},u_{20})-U_{1}(t,0,\bar{u}_{10},\bar{u}_{20})|+|U_{2}(t,0,u_{10},u_{20})-U_{2}(t,0,\bar{u}_{10},\bar{u}_{20})| \\
\leq & (1-\varpi)^{-1} \cdot(k e^{\alpha_{1} t}\cdot|P_{-}(u_{10}-\bar{u}_{10})|+ M_{B}|P_{-}(u_{20}-\bar{u}_{20})|), \quad t\leq 0.
\end{split}
\end{equation}
\end{mylemm}

\begin{proof}
  We claim the first part by means of Banach Contraction Principle.
 % Let $\mathbb{C}$ be the set of all continuous functions defined for $t\geq 0$, and for any $\zeta<\varrho$
  Let $\mathbb{BC}$ be the subset of all bounded continuous functions defined for $t\geq 0$.
  Define
\begin{equation}\label{U-1-U-2}
\begin{split}
\Phi(t):=& U_{1}(t,0,u_{10},u_{20})=e^{At}u_{10}+\int_{0}^{t} e^{A(t-s)}\cdot f( U_{1}(s,0,u_{10},u_{20}), U_{2}(s,0,u_{10},u_{20}))ds, \\
 \Psi(t):=& U_{2}(t,0,u_{10},u_{20})=e^{Bt}u_{20}.
 \end{split}
\end{equation}
  If $\mathcal{J}$ is the map defined by
\[ \mathcal{J}(\Phi,\Psi)(t)=
e^{At}\xi_1+e^{Bt}\eta_1+\int_{0}^{t} e^{A(t-s)}P_{+} f(\Phi(s),\Psi(s))ds-\int_{t}^{\infty} e^{A(t-s)}P_{-} f(\Phi(s),\Psi(s))ds,
\]
 where $\xi_1\in P_{+}(X)$ and $\eta_1\in P_{+}(Y)$.
   Then $\mathcal{J}(\Phi,\Psi)$ is continuous and bounded:
\[\begin{split}
 |\mathcal{J}(\Phi,\Psi)|\leq& ke^{-\alpha t}|\xi_1|+M_B |\eta_1|+\int_0^{\infty} |G_A (t-s)||f|_{\infty}ds \\
                 \leq& k|\xi_1|+M_B |\eta_1|+(2k|f|_{\infty}/\alpha)<\infty.
\end{split}\]
  %If we choose $\zeta<\varrho$ so that $k|\xi_1|+M_B |\eta_1|\leq (1-2k|f|_{Lip}/\alpha)\zeta$,
%then $\sup_{s\geq 0} (|\Phi(s)|+|\Psi(s)|)\leq \zeta$ implies $|\mathcal{J}(\Phi,\Psi)|\leq \zeta$.
  Hence, $\mathcal{J}$ maps $\mathbb{BC}$ into itself.
  Note that $2k|f|_{Lip}<\alpha$, for any $\Phi_1, \Phi_2, \Psi_1, \Psi_2\in \mathbb{BC}$, we have
\[\begin{split}
   |\mathcal{J}(\Phi_1,\Psi_1)-\mathcal{J}(\Phi_2,\Psi_2)| \leq& \int_0^{\infty} ke^{-\alpha|t-s|}|f|_{Lip}
                              (|\Phi_1(s)-\Phi_2(s)|+|\Psi_1(s)- \Psi_2(s))ds \\
                               \leq& (2k|f|_{Lip}/\alpha) (\|\Phi_1-\Phi_2\|+\|\Psi_1- \Psi_2\|),
\end{split}\]
which implies that $\mathcal{J}$ is a contraction mapping in $\mathbb{BC}$, that is, there exists a unique fixed point
$(\Phi^*,\Psi^*)=\mathcal{J}(\Phi^*,\Psi^*)$ such that $(\Phi^*,\Psi^*)$ is bounded for $t\geq 0$.

  Similar to the procedure just shown, $(\Phi(t),\Psi(t))$ is a mild solution of (\ref{semilinear-eq}) with the required property (ii) for
$t\leq 0$.

  We now claim the second part by means of the generalized dichotomy inequalities, (i.e., Lemmas \ref{Dichotomy-inequality-1st} and
\ref{Dichotomy-inequality-2ed}).
  Define
 \[\Omega_{i}(t):=U_{i}(t,0,u_{10},u_{20})-U_{i}(t,0,\bar{u}_{10},\bar{u}_{20}), i=1,2.\]
  It follows that
 \[\begin{split}
 \sum_{i=1}^{2}|\Omega_{i}(t)|
                  \leq &  ke^{-\alpha t}|P_{+}(u_{10}-\bar{u}_{10})| +M_{B}|P_{+}(u_{20}-\bar{u}_{20})|\\
                         &+\int_{0}^{\infty} ke^{-\alpha|t-s|}\cdot |f|_{Lip}\cdot (\sum_{i=1}^{2}|\Omega_{i}(s)|)ds.
 \end{split}\]
  Using dichotomy inequality in Lemma \ref{Dichotomy-inequality-1st}, we conclude that
\[ \sum_{i=1}^{2}|\Omega_{i}(t)|\leq (1-\varpi)^{-1}\cdot(ke^{-\alpha_{1}t}|P_{+}(u_{10}-\bar{u}_{10})|+M_{B}|P_{+}(u_{20}-\bar{u}_{20})|),
\quad t\geq 0,\]
where $\varpi=(2k|f|_{Lip}/\alpha)<1$ and positive constant $\alpha_{1}$ satisfies $\alpha_{1}=\alpha-k|f|_{Lip}/(1-\varpi)$.
\\
Finally, analogous analysis for $t\leq 0$, dichotomy inequality in Lemma \ref{Dichotomy-inequality-2ed} ends the proof.

%We will proceed the proof with two steps.\\
%{\bf Step 1.} We want to show that, for $t\geq 0$,
 %\begin{equation}\label{phi-expression}
 %\Phi(t)=e^{At}P_{+}u_{10}+\int_{0}^{t} e^{A(t-s)}P_{+}\cdot f(\Phi(s),\Psi(s))ds-\int_{t}^{\infty} e^{A(t-s)}P_{-}\cdot f(\Phi(s),\Psi(s))ds.
%\end{equation}
%We split $\Phi(t)$ into two parts: $\Phi(t)P_{+}$ and $\Phi(t)P_{-}$,
%then
%\begin{equation}\label{split}
%\begin{cases}
% \Phi(t)P_{+}=e^{At}P_{+}u_{10}+\int_{0}^{t} e^{A(t-s)}P_{+}\cdot f(\Phi(s),\Psi(s))ds, \\
% \Phi(t)P_{-}=e^{At}P_{-}u_{10}+\int_{0}^{t} e^{A(t-s)}P_{-}\cdot f(\Phi(s),\Psi(s))ds.
%\end{cases}
%\end{equation}
%It follows from the second equality of (\ref{split}) that
%\[ e^{At}P_{-}u_{10}=-\int_{0}^{t} e^{A(t-s)}P_{-}\cdot f(\Phi(s),\Psi(s))ds+ \Phi(t)P_{-}. \]
%Consequently,
%\[ P_{-}u_{10}=-\int_{0}^{t} e^{-As}P_{-}\cdot f(\Phi(s),\Psi(s))ds + \Phi(t)\cdot e^{-At}P_{-}:=-\Lambda_{1}(t)+\Lambda_{2}(t). \]
%In view of $\Phi(t)$ is bounded ($\|\Phi\|<\infty$), we have
%\[\begin{split}
%|\Lambda_{1}(t)| \leq & \int_{0}^{t}k e^{-\alpha t}\cdot|P_{-}|\cdot|f(\Phi(s),\Psi(s))|ds, \\
%|\Lambda_{2}(t)| \leq & |\Phi(t)|\cdot k e^{-\alpha t}\cdot |P_{-}|.
%\end{split}\]
%Therefore, taking $t\rightarrow +\infty$, $\|\Lambda_{1}\|\leq (k/\alpha)\cdot |f|_{\infty}$, $\|\Lambda_{2}\|\rightarrow 0$. Hence,
%\[P_{-}u_{10}=-\int_{0}^{\infty} e^{-As}P_{-}\cdot f(\Phi(s),\Psi(s))ds. \]
%Combine this with (\ref{split}), (\ref{phi-expression}) follows. We complete the proof of the Step 1.
\end{proof}
}

\newtheorem{mmyrem}{\bf{Remark}\rm}[section]
\begin{mmyrem}\label{B-hyperbolic}
 If $B$ is a hyperbolic operator, i.e., $e^{Bt}$ admits a dichotomy projection, then by the condition of Lemma \ref{bounded-Zero-solution}---$U_{2}$ is bounded, we
 obtain $U_{2} \equiv 0$. It follows from Remark \ref{remmark-dichotomy-ieq} that
 \[ \sum_{i=1}^{2}|\Omega_{i}(t)|= |\Omega_{1}(t)|\leq (1-\varpi)^{-1}ke^{-\alpha_{1}|t|}|P_{+}(u_{10}-\bar{u}_{10})|.\]
\end{mmyrem}

\begin{mmyrem}\label{expBt-decaying}
If $e^{Bt}$  is an exponential stable group, i.e., there exists $M_{B}>0, \omega_{B}>0$ such that $|e^{Bt}|\leq M_{B} e^{-\omega_{B}|t|}$.
It is obvious that $e^{Bt}$ is bounded, then
 \[\begin{split}
 \sum_{i=1}^{2}|\Omega_{i}(t)| \leq &  ke^{-\alpha |t|}|P_{+}(u_{10}-\bar{u}_{10})| +M_{B}e^{-\omega_{B}|t|}|P_{+}(u_{20}-\bar{u}_{20})|\\
                            &+\int_{0}^{\infty} ke^{-\alpha|t-s|}\cdot |f|_{Lip}\cdot (\sum_{i=1}^{2}|\Omega_{i}(s)|)ds\\
                            \leq & De^{-d |t|}(|P_{+}(u_{10}-\bar{u}_{10})|+|P_{+}(u_{20}-\bar{u}_{20})|)\\
                            &+\int_{0}^{\infty} ke^{-\alpha|t-s|}\cdot |f|_{Lip}\cdot (\sum_{i=1}^{2}|\Omega_{i}(s)|)ds.
 \end{split}\]
By using dichotomy inequality in Lemma \ref{Dichotomy-inequality-1st} and Lemma \ref{Dichotomy-inequality-2ed}, we obtain
\[ \sum_{i=1}^{2}|\Omega_{i}(t)| \leq (1-\varpi)^{-1}De^{-d_{1}|t|}\cdot(|P_{+}(u_{10}-\bar{u}_{10})|+|P_{+}(u_{20}-\bar{u}_{20})|), \]
where $D=\max\{k, M_{B}\}$, $d=\min\{ \alpha,\omega_{B} \}$ and $0<d_{1}<d$.
\end{mmyrem}

\begin{mylemm}\label{Lemma-10}
    Let $(U_{1},U_{2})$ be a mild solution of (\ref{semilinear-eq}).
    If there exist positive constants $M_{A}, M_{B}>0$ and $\omega_{A}, \omega_{B}>0$ such that $|e^{At}|\leq M_{A}e^{\omega_{A}|t|},
|e^{Bt}|\leq M_{B}e^{\omega_{B}|t|}, t\in\mathbb{R}$,
    then the inequality
\begin{equation}\label{Bellman-application}
\begin{split}
\sum_{i=1}^{2}|\Omega_{i}(t)| \leq   M_{A}e^{\omega_{A}|t|}|u_{10}-\bar{u}_{10}|+M_{B}e^{\omega_{B}|t|}|u_{20}-\bar{u}_{20}|
                                + \int_{0}^{t} |e^{A(t-s)}|\cdot |f|_{Lip}\cdot \sum_{i=1}^{2}|\Omega_{i}(s)|ds
\end{split}
\end{equation}
implies
\[\sum_{i=1}^{2}|\Omega_{i}(t)| \leq M_{c}(|u_{10}-\bar{u}_{10}|+|u_{20}-\bar{u}_{20}|) e^{(M_{c}\cdot|f|_{Lip}+\omega_{c})t}, \]
where $M_{c}=\max\{M_{A},M_{B}\}, \omega_{c}=\max\{\omega_{A},\omega_{B} \}$.
\end{mylemm}
\begin{proof}
It is obvious that
\begin{equation*}
\begin{split}
\sum_{i=1}^{2}|\Omega_{i}(t)| \leq & |e^{At}||u_{10}-\bar{u}_{10}|+|e^{Bt}||u_{20}-\bar{u}_{20}|
                                + \int_{0}^{t} |e^{A(t-s)}|\cdot |f|_{Lip}\cdot \sum_{i=1}^{2}|\Omega_{i}(s)|ds \\
                                \leq & M_{c}e^{\omega_{c}t}(|u_{10}-\bar{u}_{10}|+|u_{20}-\bar{u}_{20}|)
                                + \int_{0}^{t} M_{c}e^{\omega_{c}(t-s)}\cdot |f|_{Lip}\cdot \sum_{i=1}^{2}|\Omega_{i}(s)|ds.
\end{split}
\end{equation*}
That is
\[ \frac{\sum_{i=1}^{2}|\Omega_{i}(t)|}{e^{\omega_{c}t}}\leq M_{c}(|u_{10}-\bar{u}_{10}|+|u_{20}-\bar{u}_{20}|)
   + \int_{0}^{t} M_{c}e^{-\omega_{c}s}\cdot |f|_{Lip}\cdot \sum_{i=1}^{2}|\Omega_{i}(s)|ds.\]
By using Bellman inequality, it implies that
\[\sum_{i=1}^{2}|\Omega_{i}(t)|\leq M_{c}(|u_{10}-\bar{u}_{10}|+|u_{20}-\bar{u}_{20}|) e^{(M_{c}\cdot|f|_{Lip}+\omega_{c})t}. \]
\end{proof}

\begin{mmyrem}
In Lemma \ref{Lemma-9}, the dichotomy inequality will help us to prove the Lipschitz continuity of equivalent functions under the premise of constrained bounded solutions.
Without this premise, we can only rely on the Bellman inequality to obtain the H\"{o}lder continuity of equivalent functions.
\end{mmyrem}

\section{Proofs of main results}
    Now we are in position to prove our main results.

\subsection{  Proof of Theorem \ref{Thm-Linearization}.}
\begin{proof}
    From Lemma \ref{Lemma-7} and \ref{Lemma-8}, we have proved that $H$ and $G$ are homeomorphism.
    From Lemma \ref{Lemma-2} and \ref{Lemma-3}, we have derived that $h, g\in \mathbb{BC}(X)$.
    From Lemma \ref{Lemma-5} and \ref{Lemma-6}, $H$ sends the mild solution of semilinear evolution equation (\ref{semilinear-eq}) onto the mild solution
of linear evolution equation (\ref{linear-eq}) and vice versa.
    Hence, (\ref{semilinear-eq}) and (\ref{linear-eq}) are topologically conjugated. This completes the proof of Theorem \ref{Thm-Linearization}.
\end{proof}

\subsection{  Proof of Theorem \ref{Thm-Regulaity}}
 % Corresponding to the two cases of Theorem \ref{Thm-Regulaity},
We split the proof of Theorem \ref{Thm-Regulaity} into two   steps.
\begin{proof}
%\begin{description}
%\item[\bf{Proof ($C_{1}$).}]  
The aim in this part is to claim that $H$ is Lipschitzian, and $G$ is H\"{o}lder continuous.\\
{\bf Step 1-1.}
 {\color{blue}
 We will use  dichotomy inequality in Lemma \ref{Lemma-9} to prove the Lipshcitz continuity of the linearising map.
  We show that $|H(u)-H(\bar{u})|\leq p_{1}|u-\bar{u}|$, where $p_{1} \geq 1$ is a constant.
  From Lemma \ref{Lemma-2}, it follows that
\[ h(\xi,\eta)= -\int_{\mathbb{R}}G_{A}(-s)f(U_{1}(s,0,\xi,\eta), U_{2}(s,0,\xi,\eta))ds,\]
which is equivalent to
  \[ P_+ h(\xi,\eta)=-\int_{-\infty}^{0} e^{-As}P_{+}f(U_{1}(s,0,\xi,\eta),U_{2}(s,0,\xi,\eta))ds\]
and
  \[ P_{-} h(\xi,\eta)=\int_{0}^{\infty} e^{-As}P_{-}f(U_{1}(s,0,\xi,\eta),U_{2}(s,0,\xi,\eta))ds.\]
  Thus we get
\begin{equation*}
  \begin{split}
   R_1\triangleq & P_+ h(\xi,\eta)- P_+ h(\bar{\xi},\bar{\eta}) \\
        =&\int_{-\infty}^{0} e^{-As}P_{+} [f(U_{1}(s,0,\bar{\xi},\bar{\eta}),
U_{2}(s,0,\bar{\xi},\bar{\eta}))-f(U_{1}(s,0,\xi,\eta),U_{2}(s,0,\xi,\eta))]ds, \\
  R_2\triangleq & P_- h(\xi,\eta)- P_- h(\bar{\xi},\bar{\eta}) \\
  =& \int_{0}^{\infty} e^{-As}P_{-} [f(U_{1}(s,0,\bar{\xi},\bar{\eta}),
U_{2}(s,0,\bar{\xi},\bar{\eta}))-f(U_{1}(s,0,\xi,\eta),U_{2}(s,0,\xi,\eta))]ds.
 \end{split}
\end{equation*}
  %Since the Cauchy problem of equation (\ref{semilinear-eq}) is bounded,
  In view of Lemma \ref{Lemma-9}, for any initial condition on $P_{+}(X)$ and $P_{+}(Y)$ (or $P_{-}(X), P_{-}(Y)$) of (\ref{semilinear-eq})
is bounded on semiaxis $[0,\infty)$ (or $(-\infty,0]$).
  Then by using Lemma \ref{Lemma-9}, part $t\leq 0$, we deduce that
\[\begin{split}
  |R_1|\leq & \int_{-\infty}^{0} k e^{\alpha s} \cdot|f|_{Lip}\cdot (\sum_{i=1}^{2}|\Omega_{i}(s)|)ds\\
   \leq& \int_{-\infty}^{0} k e^{\alpha s} \cdot|f|_{Lip}\cdot \frac{1}{1-\varpi}
             (k e^{-\alpha_{1} s}\cdot|P_{+}(\xi-\bar{\xi})| + M_{B}|P_{+}(\eta-\bar{\eta})|)ds,
\end{split}\]
and similarly, by using Lemma \ref{Lemma-9}, part $t\geq 0$,
\[\begin{split}
  |R_2|\leq  \int_{0}^{\infty} k e^{-\alpha s} \cdot|f|_{Lip}\cdot \frac{1}{1-\varpi}
             (k e^{\alpha_{1} s}\cdot|P_{-}(\xi-\bar{\xi})| + M_{B}|P_{-}(\eta-\bar{\eta})|)ds.
\end{split}\]
  Hence we conclude that
\[\begin{split}
 |R_1|+|R_2| \leq & \int_{-\infty}^{0} k e^{\alpha s} \cdot|f|_{Lip}\cdot \frac{1}{1-\varpi}
             (k e^{-\alpha_{1} s}\cdot|P_{+}(\xi-\bar{\xi})| + M_{B}|P_{+}(\eta-\bar{\eta})|)ds \\
     &+\int_{0}^{\infty} k e^{-\alpha s} \cdot|f|_{Lip}\cdot \frac{1}{1-\varpi}
             (k e^{\alpha_{1} s}\cdot|P_{-}(\xi-\bar{\xi})| + M_{B}|P_{-}(\eta-\bar{\eta})|)ds \\
     \leq & \frac{k^{2}\cdot|f|_{Lip}}{(\alpha-\alpha_{1})(1-\varpi)}\cdot[|P_{+}(\xi-\bar{\xi})|+|P_{-}(\xi-\bar{\xi})|] \\
     &+ \frac{k\cdot|f|_{Lip}\cdot M_{B}}{\alpha(1-\varpi)} \cdot[|P_{+}(\eta-\bar{\eta})|+|P_{-}(\eta-\bar{\eta})|] \\
   % \int_{\mathbb{R}} k e^{-\alpha|s|}\cdot |f|_{Lip}\cdot (\sum_{i=1}^{2}|\Omega_{i}(s)|)ds\\
    %      \leq  & \int_{\mathbb{R}} k e^{-\alpha|s|}\cdot |f|_{Lip}\cdot \frac{1}{1-\varpi}
     %        (k e^{-\alpha_{1} |s|}\cdot|\xi-\bar{\xi}| + M_{B}|\eta-\bar{\eta}|)ds\\
          \leq& \max\left\{\frac{2k^{2}\cdot|f|_{Lip}}{(\alpha_{1}+\alpha)(1-\varpi)}, \frac{2k\cdot|f|_{Lip}\cdot M_{B}}{\alpha(1-\varpi)} \right\}(|\xi-\bar{\xi}|+|\eta-\bar{\eta}|).
\end{split}\]
  By the definition of $H(u)$, $u=(u_{1},u_{2})^{T}$,
\[\begin{split}
 | H(u)-H(\bar{u})| \leq & |u-\bar{u}|+\max\left\{\frac{2k\cdot|f|_{Lip}}{\alpha(1-\varpi)}, \frac{k^{2}\cdot|f|_{Lip}}{\alpha(1-\varpi)} \right\}
                               |u-\bar{u}| \\
                        \leq & \left(1+ \max\left\{\frac{2k^{2}\cdot|f|_{Lip}}{(\alpha_{1}+\alpha)(1-\varpi)}, \frac{2k\cdot|f|_{Lip}\cdot M_{B}}{\alpha(1-\varpi)} \right\}\right)|u-\bar{u}| \\
                        :=& p_{1} |u-\bar{u}|.
\end{split}\]
  This completes the proof of Step 1-1.
}
\newtheorem{myremark}{\bf{Remark}\rm}[section]
\begin{myremark}
  If $B$ is a hyperbolic operator, in view of Remark \ref{B-hyperbolic}, we obtain
\[\begin{split}
 | h(\xi,\eta)- h(\bar{\xi},\bar{\eta})| \leq &
                             \int_{\mathbb{R}} ke^{\alpha|s|}\cdot |f|_{Lip}\cdot \frac{ke^{-\alpha_{1}|s|}|\xi-\bar{\xi}|}{1-\varpi}ds \\
                         \leq & 2\int_{0}^{\infty} \frac{k^{2}|f|_{Lip}}{1-\varpi}|\xi-\bar{\xi}|\cdot e^{(\alpha_1-\alpha)s}ds \\
                         \leq & \frac{2k^{2}|f|_{Lip}}{(1-\varpi)(\alpha-\alpha_{1})}\cdot |\xi-\bar{\xi}|.
\end{split}\]
  Hence, $h$ is Lipschitzian, consequently, $H=x+h$ is also Lipschitzian.
\end{myremark}

\begin{myremark}
  If $e^{Bt}$ is an exponential stable group, it follows from Remark \ref{expBt-decaying} that
\[\begin{split}
 | h(\xi,\eta)- h(\bar{\xi},\bar{\eta})| \leq &
             \int_{\mathbb{R}} ke^{\alpha|s|}\cdot |f|_{Lip}\cdot \frac{D(|\xi-\bar{\xi}|+|\eta-\bar{\eta}|)e^{-d_{1}|s|}}{1-\varpi}ds \\
           \leq &  \frac{2kD|f|_{Lip}}{(1-\varpi)(\alpha-d_{1})}\cdot (|\xi-\bar{\xi}|+|\eta-\bar{\eta}|).
\end{split}\]
  It is easy to see that $H$ is Lipschitzian.
\end{myremark}
{\bf Step 1-2.} We are going to prove that the inverse $G=H^{-1}$ is H\"{o}lder continuous.
  By Lemma \ref{Lemma-3}, we know that $g(\xi,\eta)$ is a fixed point of the following map $\mathcal{T}$
\begin{equation}\label{g-expression}
\begin{split}
 (\mathcal{T}z)(0)=& \int_{\mathbb{R}} G_{A}(-s)\cdot f(V_{1}(s,0,\xi,\eta)+z(s),(V_{2}(s,0,\xi,\eta))ds \\
    =& \int_{\mathbb{R}} G_{A}(-s)\cdot f(e^{As}\xi+z(s),e^{Bs}\eta)ds.
\end{split}
\end{equation}
  Let $g_{0}(\xi,\eta)\equiv 0$, and by recursion define
\[ g_{m+1}(\xi,\eta)= \int_{\mathbb{R}} G_{A}(-s)\cdot f(e^{As}\xi+g_{m}(\xi,\eta),e^{Bs}\eta)ds.\]
  It is not difficult to show that
\[ g_{m}(\xi,\eta)\rightarrow  g(\xi,\eta), \quad \mathrm{as }\ \ m\rightarrow +\infty, \]
  uniformly with respect to $\xi,\eta$.
\\
  Note that $g_{0}(\xi,\eta)=g_{0}(t,(t,e^{At}\xi,e^{Bt}\eta))\equiv 0$. Thus, by induction, it is clear that for all $m (m\in\mathbb{N})$,
$g_{m}(\xi,\eta)=g_{m}(t,(t,e^{At}\xi,e^{Bt}\eta))$. Choose $p>0$ sufficiently large and $q>0$ sufficiently small such that
  \[\begin{cases}
    p>(8k/\alpha)\cdot|f|_{\infty}+\frac{4k|f|_{Lip}\cdot M_{c}}{\omega_{c}-\alpha}, \\
    q\omega_{c}<\alpha,\\
    0<\frac{2k|f|_{Lip}\cdot M_{c}^{q}}{\alpha-q\omega_{c}}<\frac{1}{2}.
  \end{cases}\]
  Then, we want to show that
\begin{equation}\label{g-assume-holder}
|g_{m}(\xi,\eta)-g_{m}(\bar{\xi},\bar{\eta})| \leq p(|\xi-\xi|+|\eta-\bar{\eta}|)^{q}.
\end{equation}
  Obviously, inequality (\ref{g-assume-holder}) holds if $m=0$. Now making the inductive assumption that (\ref{g-assume-holder}) holds.
  From (\ref{g-expression}), it follows that
\[\begin{split}
  &g_{m+1}(\xi,\eta)-g_{m+1}(\bar{\xi},\bar{\eta})\\
%=& \int_{\mathbb{R}} G_{A}(-s)\cdot [f(e^{As}\xi+g_{m}(\xi,\eta),e^{Bs}\eta)\\
% &- f(e^{As}\bar{\xi}+g_{m}(\bar{\xi},\bar{\eta}),e^{Bs}\bar{\eta})]ds. \\
=& 2\int_{0}^{\infty} G_{A}(-s)\cdot [f(e^{As}\xi+g_{m}(\xi,\eta),e^{Bs}\eta)
 -f(e^{As}\bar{\xi}+g_{m}(\bar{\xi},\bar{\eta}),e^{Bs}\bar{\eta})]ds \\
=& 2\int_{0}^{\tau_{2}} G_{A}(-s)\cdot [f(e^{As}\xi+g_{m}(\xi,\eta),e^{Bs}\eta)
 -f(e^{As}\bar{\xi}+g_{m}(\bar{\xi},\bar{\eta}),e^{Bs}\bar{\eta})]ds \\
 &+2\int_{\tau_{2}}^{\infty} G_{A}(-s)\cdot [f(e^{As}\xi+g_{m}(\xi,\eta),e^{Bs}\eta)
 -f(e^{As}\bar{\xi}+g_{m}(\bar{\xi},\bar{\eta}),e^{Bs}\bar{\eta})]ds \\
\triangleq& 2(I_{1}+I_{2}),
\end{split}\]
  where $\tau_{2}=\frac{1}{\omega_{c}} \cdot \ln \frac{1}{|\xi-\xi|+|\eta-\bar{\eta}|}$.
  In view of $\alpha< \omega_{c}$, we have
\[\begin{split}
|I_{2}| \leq & \int_{\tau_{2}}^{\infty} ke^{-\alpha s}\cdot 2|f|_{\infty} ds \leq (2k/\alpha)\cdot |f|_{\infty}\cdot e^{-\alpha \tau_{2}}
 =(2k/\alpha)\cdot |f|_{\infty}\cdot [|\xi-\xi|+|\eta-\bar{\eta}|]^{\frac{\alpha}{\omega_{c}}}.
\end{split}\]
  Notice that $A, B$ are the generators of $C_{0}$-groups $e^{At}, e^{Bt}$ in $X$ with $|e^{At}|\leq M_{A} e^{\omega_{A}|t|}$, $|e^{Bt}|\leq M_{B} e^{\omega_{B}|t|}$, respectively. Then we have the following estimates
\[
|e^{At}x-e^{At}\bar{x}|\leq M_{A} |x-\bar{x}|e^{\omega_{A}|t|},\; |e^{Bt}x-e^{Bt}\bar{x}|\leq M_{B} |x-\bar{x}|e^{\omega_{B}|t|},
\]
  where $M_{A},M_{B}\geq 1, \omega_{A}, \omega_{B}>0$.
  In view of Lemma \ref{Lemma-10}, we have
\[\begin{split}
 |g_{m}(\xi,\eta)-g_{m}(\bar{\xi},\bar{\eta})| =& |g_{m}(t,(t,e^{At}\xi,e^{Bt}\eta))-g_{m}(t,(t,e^{At}\bar{\xi},e^{Bt}\bar{\eta}))| \\
       \leq & p[|e^{At}\xi-e^{At}\bar{\xi}|+|e^{Bt}\eta-e^{Bt}\bar{\eta}|]^{q} \\
       \leq & p[M_{A}e^{\omega_{A}|t|}|\xi-\bar{\xi}|+M_{B}e^{\omega_{B}|t|}|\eta-\bar{\eta}|]^{q} \\
       \leq & p[M_{c}e^{\omega_{c}|t|}(|\xi-\xi|+|\eta-\bar{\eta}|)]^{q}.
\end{split}\]
  Therefore,
\[\begin{split}
 |I_{1}| \leq & \int_{0}^{\tau_{2}} ke^{-\alpha s}\cdot |f|_{Lip}\cdot \{M_{A}e^{\omega_{A}s}|\xi-\bar{\xi}|
       + M_{B}e^{\omega_{B}s}|\eta-\bar{\eta}|\\
       &+  p[M_{c}e^{\omega_{c}s}(|\xi-\xi|+|\eta-\bar{\eta}|)]^{q}\}ds \\
       \leq & \int_{0}^{\tau_{2}} ke^{-\alpha s}\cdot |f|_{Lip}\cdot \{M_{c}e^{\omega_{c}s}(|\xi-\xi|+|\eta-\bar{\eta}|)\\
       & + p[M_{c}e^{\omega_{c}s}(|\xi-\xi|+|\eta-\bar{\eta}|)]^{q}\} ds \\
       \leq &  \int_{0}^{\tau_{2}} k|f|_{Lip}\cdot M_{c} e^{(\omega_{c}-\alpha)s}  (|\xi-\xi|+|\eta-\bar{\eta}|)ds\\
       &+ \int_{0}^{\tau_{2}} k|f|_{Lip}\cdot p M_{c}^{q}  e^{(\omega_{c}q-\alpha)s}  (|\xi-\xi|+|\eta-\bar{\eta}|)^{q} ds \\
       \leq & \frac{k|f|_{Lip}\cdot M_{c}(|\xi-\xi|+|\eta-\bar{\eta}|) }{\omega_{c}-\alpha} \cdot e^{(\omega_{c}-\alpha)s}|^{s=\tau_{2}}_{s=0} \\
       &+ \frac{k|f|_{Lip}\cdot p M_{c}^{q}(|\xi-\xi|+|\eta-\bar{\eta}|)^{q} }{\alpha-q\omega_{c}} \cdot e^{(q\omega_{c}-\alpha)s}|^{s=0}_{s=\tau_{2}}.
\end{split}\]
  Note that $\omega_{c}-\alpha>0, q\omega_{c}-\alpha<0$ $(0<q<1)$, then
\[\begin{split}
 |I_{1}| \leq  \frac{k|f|_{Lip}\cdot M_{c}}{\omega_{c}-\alpha} (|\xi-\xi|+|\eta-\bar{\eta}|)^{\frac{\alpha}{\omega_{c}}}
           + \frac{k|f|_{Lip}\cdot p M_{c}^{q}}{\alpha-q\omega_{c}}(|\xi-\xi|+|\eta-\bar{\eta}|)^{q}.
\end{split}\]
  Therefore, we obtain
\[\begin{split}
   &|g_{m+1}(\xi,\eta)-g_{m+1}(\bar{\xi},\bar{\eta})| \\
   \leq & 2(|I_{1}|+|I_{2}|) \\
     \leq&    \frac{2k|f|_{Lip}\cdot M_{c}}{\omega_{c}-\alpha} (|\xi-\xi|+|\eta-\bar{\eta}|)^{\frac{\alpha}{\omega_{c}}}
    + \frac{2k|f|_{Lip}\cdot p M_{c}^{q}}{\alpha-q\omega_{c}}(|\xi-\xi|+|\eta-\bar{\eta}|)^{q} \\
    &+(4k/\alpha)\cdot |f|_{\infty}\cdot [|\xi-\xi|+|\eta-\bar{\eta}|]^{\frac{\alpha}{\omega_{c}}} \\
\leq& p \cdot [|\xi-\xi|+|\eta-\bar{\eta}|]^{q}.
\end{split}\]
  It implies that
\[ |g(\xi,\eta)-g(\bar{\xi},\bar{\eta})| \leq p \cdot [|\xi-\bar{\xi}|+|\eta-\bar{\eta}|]^{q},\quad \mathrm{as}\quad  m\rightarrow\infty.\]
  By the definition of $G(v)$, $v=(v_{1},v_{2})^{T}$, we obtain that
\[ |G(v)-G(\bar{v})| \leq (1+p)\cdot |v-\bar{v}|^{q}:=p_{2}\cdot |v-\bar{v}|^{q}.\]
  This completes the proof of Step 1-2.

\begin{myremark}\label{g-holder}
  We now explain that $G$ cannot be improved to be Lipschitzian.
  In fact, in proving the H\"{o}lder regularity of $G$, we divide $g_{m+1}(\xi,\eta)-g_{m+1}(\bar{\xi},\bar{\eta})$ into two parts: $I_{1},I_{2}$.
  If we directly prove that $G$ is Lipschitzian, then the   contradiction appears as follows:
\[\begin{split}
 |g_{m+1}(\xi,\eta)-g_{m+1}(\bar{\xi},\bar{\eta})| \leq & \int_{0}^{\infty} ke^{-\alpha s}\cdot |f|_{Lip}\cdot \{M_{A}e^{\omega_{A}s}|\xi-\bar{\xi}|
       + M_{B}e^{\omega_{B}s}|\eta-\bar{\eta}|\\
       &+  p[M_{c}e^{\omega_{c}s}(|\xi-\xi|+|\eta-\bar{\eta}|)]^{q}\}ds \\
       \leq &  \int_{0}^{\infty} k|f|_{Lip}\cdot M_{c} e^{(\omega_{c}-\alpha)s}  (|\xi-\xi|+|\eta-\bar{\eta}|)ds\\
       &+ \int_{0}^{\infty} k|f|_{Lip}\cdot p M_{c}^{q}  e^{(\omega_{c}q-\alpha)s}  (|\xi-\xi|+|\eta-\bar{\eta}|)^{q} ds.
\end{split}\]
  Since $\omega_{c}-\alpha>0$, the integral $\int_{0}^{\infty}e^{(\omega_{c}-\alpha)s}ds$ is divergent
in the above first right-side estimation.
Hence, $G$ cannot be improved to be Lipschitzian.
\end{myremark}
\end{proof}

\subsection{  Proof of Theorem \ref{Thm-Regulaity2} }

\begin{proof}
%\item[\bf{Proof ($C_{2}$).}]
  The goal in this part is to claim that $H$ and $G$ are H\"{o}lder continuous based on Bellman inequality.\\
  {\bf Step 2-1.}
  %We first prove that $H$ is H\"{o}lder continuous.
  %Without the premise of constrained bounded solutions, the dichotomy inequality in Lemma \ref{Lemma-9} is not valid for the proof of
%Theorem \ref{Thm-Regulaity}.
  From the points of our mechanism, we say that the linearising map is merely H\"{o}lder continuous based on the Bellman inequality in Lemma \ref{Lemma-10}.
   Without loss of generality, assume that $|\xi-\bar{\xi}|+|\eta-\bar{\eta}|<1$.
  Set $\tau_{1}=\frac{1}{\omega_{c}+M_{c}|f|_{Lip}}\cdot \ln \frac{1}{|\xi-\bar{\xi}|+|\eta-\bar{\eta}|}$, we see that
\[\begin{split}
 &h(\xi,\eta)- h(\bar{\xi},\bar{\eta})\\
 =& \int_{\mathbb{R}}G_{A}(-s) [f(U_{1}(s,0,\bar{\xi},\bar{\eta}), U_{2}(s,0,\bar{\xi},\bar{\eta}))
                                      -f(U_{1}(s,0,\xi,\eta),U_{2}(s,0,\xi,\eta))]ds\\
 =&2\int_{0}^{\tau_{1}} G_{A}(-s) [f(U_{1}(s,0,\bar{\xi},\bar{\eta}), U_{2}(s,0,\bar{\xi},\bar{\eta}))
                                      -f(U_{1}(s,0,\xi,\eta),U_{2}(s,0,\xi,\eta))]ds\\
                                     &+ 2\int_{\tau_{1}}^{\infty} G_{A}(-s) [f(U_{1}(s,0,\bar{\xi},\bar{\eta}), U_{2}(s,0,\bar{\xi},\bar{\eta}))
                                      -f(U_{1}(s,0,\xi,\eta),U_{2}(s,0,\xi,\eta))]ds\\
                                      \triangleq& 2(J_{1}+J_{2}),
\end{split}\]
  Note that $\alpha< \omega_{c}$, we have
\[\begin{split}
|J_{2}| \leq & \int_{\tau_{1}}^{\infty} ke^{-\alpha s}\cdot 2|f|_{\infty} ds \leq (2k/\alpha)\cdot |f|_{\infty}\cdot e^{-\alpha \tau_{1}}\\
 =& (2k/\alpha)\cdot |f|_{\infty}\cdot [|\xi-\xi|+|\eta-\bar{\eta}|]^{\frac{\alpha}{\omega_{c}+M_{c}|f|_{Lip}}}.
\end{split}\]
  Since $M_{c}\cdot |f|_{Lip}+\omega_{c}>\alpha$, and it follows from Lemma \ref{Lemma-10} that
\[\begin{split}
 |J_{1}| \leq & \int_{0}^{\tau_{1}} ke^{-\alpha s}\cdot |f|_{Lip}\cdot M_{c}\cdot (|\xi-\xi|+|\eta-\bar{\eta}|)
                     e^{(M_{c}\cdot|f|_{Lip}+\omega_{c})s}ds\\
            \leq & k\cdot|f|_{Lip}\cdot M_{c}\cdot (|\xi-\xi|+|\eta-\bar{\eta}|)\int_{0}^{\tau_{1}}e^{(M_{c}\cdot|f|_{Lip}+\omega_{c}-\alpha)s}ds\\
            \leq & \frac{k\cdot|f|_{Lip}\cdot M_{c}}{M_{c}\cdot|f|_{Lip}+\omega_{c}-\alpha}\cdot
               [|\xi-\xi|+|\eta-\bar{\eta}|]^{1-\frac{M_{c}\cdot|f|_{Lip}+\omega_{c}-\alpha}{\omega_{c}+M_{c}|f|_{Lip}}} \\
             \leq & \frac{k\cdot|f|_{Lip}\cdot M_{c}}{M_{c}\cdot|f|_{Lip}+\omega_{c}-\alpha}\cdot[|\xi-\xi|+|\eta-\bar{\eta}|]^{\frac{\alpha}{\omega_{c}+M_{c}|f|_{Lip}}}.
\end{split}\]
  Therefore,
\[\begin{split}
|h(\xi,\eta)- h(\bar{\xi},\bar{\eta})| \leq & 2(|J_{1}|+|J_{2}|) \\
                        \leq & \left[(4k/\alpha) |f|_{\infty}+ \frac{2k|f|_{Lip} M_{c}}{M_{c}|f|_{Lip}+\omega_{c}-\alpha}\right]
                       [|\xi-\xi|+|\eta-\bar{\eta}|]^{\frac{\alpha}{\omega_{c}+M_{c}|f|_{Lip}}}.
\end{split}\]
  Further, it implies that $H$ is H\"{o}lder continuous, i.e.,
\[ |H(u)-H(\bar{u})| \leq \tilde{p}_{1}\cdot |u-\bar{u}|^{\tilde{q}}.\]
{\bf Step 2-2.}
  We  show that $G$ is also H\"{o}lder continuous.
  From Step 1-2, we immediately obtain it.
%\end{description}
Hence, we complete the proof of Theorem \ref{Thm-Regulaity2}.
\end{proof}

\subsection{Proof of Theorem \ref{Local-theorem}}
\begin{proof}
  We just verify that it satisfies all the conditions of Theorem \ref{Thm-Linearization}.
  Since $\delta$ is said to be locality radius of a spherical neighborhood $\mathcal{B}(0,\delta)$,
it is clear to see that $L(\sqrt{2}\delta,\sqrt{2}\delta)$ is a small
constant and
\[\begin{split}
 |f_{\delta}|_{Lip}=& 9L(\sqrt{2}\delta,\sqrt{2}\delta), \\
 |f_{\delta}|_{\infty}=& 2\sqrt{2}\delta\cdot L(\sqrt{2}\delta,\sqrt{2}\delta)<\infty.
\end{split}\]
In addition, from the other conditions of Theorem \ref{Local-theorem} we can obtain its local linearization in $\overline{\mathcal{B}}(0,\delta)$.
\end{proof}

\section{Applications}
  We consider the coupled $parabolic$ equation and $ordinary$ differenatial equation:
\begin{equation}\label{App-1}
  \begin{cases}
   \frac{\partial u}{\partial t}=\frac{\partial^{2} u}{\partial x^{2}}+f(u,v), \; (0<x<\pi),\\
   \frac{dv}{dt}=Bv,\\
   u(0,t)=0, \; u(\pi,t)=0,\\
   u(x,0)=u_{0},\; v(0)=v_{0}.
  \end{cases}
\end{equation}
  It is easy to see that $B$ is the generator of $C_{0}$-group $e^{Bt}$ on Banach space $X$, and
\[ D(B)=\{ v_{0}\in X|\;\mathrm{ there}\;\mathrm{ exists}\; \mathrm{a}\; \mathrm{constant}\; h\;
           \mathrm{such}\; \mathrm{that}\; \lim\limits_{h\rightarrow 0}\frac{e^{Bh}v_{0}-v_{0}}{h}\; \mathrm{holds}\}.\]
  Define the linear operator $A$ by
\[ A\psi(x):= -\frac{\partial^{2}\psi(x)}{\partial x^{2}}, \; 0<x<\pi,\]
where $\psi$ is a smooth function on $[0,\pi]$ with $\psi(0)=0, \psi(\pi)=0$.
  Using Friedrichs theorem in \cite{Henry-book},  $A$ can be extended to a self-adjoint densely defined linear operator in $L^{2}(0,\pi)$.
  In this cases,
\[ D(A)=\{\psi\in L^{2}(0,\pi)|A\psi\in L^{2}(0,\pi)=H_{0}^{1}(0,\pi)\cap H^{2}(0,\pi)\},\]
and the spectrum $\sigma(A)$ of $A$ consists only of simple eigenvalues $\lambda_{n}=n^{2}, (n=1,2,\cdots)$ with corresponding eigenfunctions
$\psi_{n}(x)=(2/\pi)^{\frac{1}{2}}\sin nx$.
  Then the coupled system can be rewritten as a differential equation on Banach space:
\begin{equation}\label{App-2}
  \partial_{t} u=-Au+f(u, e^{Bt}v_{0}).
\end{equation}
  Now, we give a version of linearization theorem for equations (\ref{App-1}).
\newtheorem{app-theorem}[exam]{Theorem}
\begin{app-theorem}\label{Application-theorem}
  Assume that $\partial_{t}u=-Au$ admits an exponential dichotomy. If the nonlinearity $f$ is bounded and Lipschitzian, and $|f|_{Lip}<1$.
  Then equations (\ref{App-1}) is topologically conjugated to its linear parts.
\end{app-theorem}
\begin{proof}
  We just verify that this theorem satisfies all of the conditions for Theorem \ref{Thm-Linearization}.
  It is known that $f$ is bounded and Lipschitzian.
  It is easy to see that equation (\ref{App-2}) has a global existence and uniqueness result, i.e.,
\begin{equation}\label{App-IE}
  u(t)=e^{-At}u_{0}+\int_{0}^{t}e^{-A(t-s)}\cdot f(u(s), e^{Bs}v_{0}) ds,
\end{equation}
which comes from the variation of constants formula.
  Since the spectrum $\sigma(A)$ consists of simple eigenvalues $\lambda_{n}$, we clear that $A$ admits a dichotomy projection (in fact, we know that
it is a trivial dichotomy).
  Without loss of generality, taking $4k\alpha^{-1}=|\max\{-\lambda_{n}\}|=1$.
  Hence, condition (\ref{third-linear-condition}) can be rewritten as: $|\max\{-\lambda_{n}\}|\cdot |f|_{Lip}=|f|_{Lip}<1$.
  Therefore, all the conditions of Theorem \ref{Thm-Linearization} are satisfied.
  We say that equations (\ref{App-1}) is topologically conjugated to its linear equations.
\end{proof}

\noindent {\bf Example.}  Consider the Hodekin-Huxley equations for the nerve axon:
\begin{equation}\label{App-3}
  \begin{cases}
  C\frac{\partial V}{\partial t}=\frac{1}{r_{e}+r_{i}}\frac{\partial^{2} V}{\partial x^{2}}-g_{k} n^{4}(V-E_{k})-g_{Na} m^{3}h(V-E_{Na}),\\
  \frac{\partial n}{\partial t}=\alpha_{n}(v)(1-n)-\beta_{n}(v)n, \\
  \frac{\partial m}{\partial t}=\alpha_{m}(v)(1-m)-\beta_{m}(v)m, \\
  \frac{\partial h}{\partial t}=\alpha_{h}(v)(1-h)-\beta_{h}(v)h,
  \end{cases}
\end{equation}
where $V$ denote the electrical potential and $C$ denote the membrane capacitance.
  The quantities $m, n, h$, which vary between $0$ and $1$, and describe the changes in the conductance of the axon membrane for sodium ($Na$)
and potassium ($K$).
  One can see Cole \cite{Cole} for more details.
  Here we take $\alpha_{n}, \beta_{n}$, etc. are all positive constants.
  Therefore, equations (\ref{App-3}) is a coupled equations with a parabolic equation and ODEs, it can be written as:
\begin{equation}\label{App-4}
  \begin{cases}
  \frac{\partial V}{\partial t}=-AV-g_{k} n^{4}(V-E_{k})-g_{Na} m^{3}h(V-E_{Na}),\\
  \frac{\partial n}{\partial t}=\gamma_{n}n-\alpha_{n}, \\
  \frac{\partial m}{\partial t}=\gamma_{m}m-\alpha_{m}, \\
  \frac{\partial h}{\partial t}=\gamma_{h}h-\alpha_{h},
  \end{cases}
\end{equation}
where $-AV=\frac{C^{-1}}{r_{e}+r_{i}}\frac{\partial^{2} V}{\partial x^{2}}$, and $C, g_{k}, E_{k}, g_{Na},E_{Na},$ etc. are constants.
  Define $f(V,n,m,h):=-g_{k} n^{4}(V-E_{k})-g_{Na} m^{3}h(V-E_{Na})$,
since $m, n, h\in [0,1]$, we can check that $f$ is Lipschitzian.
 Further if $f$ is bounded.
  Since $-A$ has a dichotomy projection, $\gamma_{n}, \gamma_{m}, \gamma_{h}$ are constants.
  It follows from Theorem \ref{Application-theorem} that equations (\ref{App-4}) can be completely linearised.

\end{document}